\documentclass{article}
\usepackage[utf8]{inputenc}
\usepackage{authblk}
\usepackage{hyperref}
\usepackage{enumitem}

\usepackage[a4paper, total={6in, 9in}]{geometry}

\usepackage{array}
\usepackage{amsthm}
\usepackage{amsmath}
\usepackage{amssymb}
\usepackage{amsfonts}
\usepackage[T1]{fontenc}
\usepackage{xfrac}
\usepackage[english]{babel}
\usepackage{mathtools}

\theoremstyle{plain}
\newtheorem{theorem}{Theorem}[section]
\newtheorem{lemma}[theorem]{Lemma}
\newtheorem{proposition}[theorem]{Proposition}
\newtheorem{corollary}[theorem]{Corollary}

\theoremstyle{definition}
\newtheorem{definition}[theorem]{Definition}

\theoremstyle{remark}
\newtheorem{example}[theorem]{Example}
\newtheorem{remark}[theorem]{Remark}

\newcommand{\Z}{\mathbb{Z}}
\newcommand{\Q}{\mathbb{Q}}
\newcommand{\R}{\mathbb{R}}
\newcommand{\C}{\mathbb{C}}
\newcommand{\F}{\mathbb{F}}

\newcommand{\Zp}{\mathbb{Z}_p}
\newcommand{\Qp}{\mathbb{Q}_p}
\newcommand{\Cp}{\mathbb{C}_p}
\newcommand{\quot}[2]{\sfrac{#1}{#2}}
\DeclareMathOperator{\Fix}{Fix}
\DeclareMathOperator{\Cay}{Cay}
\DeclareMathOperator{\Stab}{Stab}
\DeclareMathOperator{\tr}{tr}
\DeclareMathOperator{\Gal}{Gal}

\DeclareMathOperator{\Mat}{Mat}
\DeclareMathOperator{\Ind}{Ind}
\DeclareMathOperator{\Tame}{Tame}
\DeclareMathOperator{\ord}{ord}
\DeclareMathOperator{\orb}{orb}

\title{The Algebraicity Problem for Hard-Core Entropy Constants on the Discrete Hypertori}
\author{Yotam Svoray}
\date{ }

\AtEndDocument{\bigskip{\footnotesize%
  \textsc{Department of Mathematics, Weizmann Institute of Science, Israel} \par
   \textit{E-mail address}: \texttt{yotam.svoray@huji.mail.ac.il}
}}

\begin{document}

\maketitle

\begin{abstract}
 We use tools and techniques from $p$-adic analysis and algebraic number theory to study the algebraicity of the hard square entropy constant and its high dimensional analogues. Specifically, we study arithmetic properties of $a_d(n)$, the number of independent sets in the $d$-dimensional discrete torus, and the associated entropy constants $\kappa_d=\lim_{n\to\infty}a_d(n)^{1/n^d}$. It is not known whether $\kappa_d$ is algebraic or transcendental for $d>1$. Using the fact that the sequence $a_d(p^k)$ converges $p$-adically for every prime $p$ and other arithmetic facts, we present a collection of criteria for the algebraicity of $\kappa_d$ and bound the number of possible values of prime powers $p^k$ for which $a_d(p^k)=\kappa_d^{p^{kd}}$. 
\end{abstract}

\tableofcontents

\section{Introduction}

The hard-core model is one of the basic lattice models of statistical
mechanics and symbolic dynamics. At activity $1$, its finite-volume partition
function is simply the number of independent sets in the underlying graph.
For integers $d,n\geq1$, let $T_d(n)$ denote the $d$-dimensional discrete
torus of side length $n$, and let $a_d(n)$ be its number of independent
sets\footnote{An independent set in a graph is a set of vertices, no two of which are connected by an edge.}. The associated entropy constant is $\kappa_d=\lim_{n\to\infty}a_d(n)^{1/n^d}$. The $\log\kappa_d$ is the topological entropy of the
$d$-dimensional hard-core shift.\\

The purpose of this paper is to investigate arithmetic structures hidden in
the finite-volume counts $a_d(n)$ and to determine what these structures
imply about the algebraicity of $\kappa_d$. The sequence $a_d(n)$ is governed
simultaneously by three kinds of structure: The first is the asymptotic
structure of the finite-volume counts, which produces the "archimedean" limit $a_d(n)^{\frac{1}{n^d}} \to \kappa_d$, the second
is the algebraic structure of transfer matrices, which expresses
finite-volume counts through traces and algebraic eigenvalues, and the third
is the arithmetic structure arising from translation actions on periodic
configurations, which produces strong congruences at prime powers. The
central theme of the paper is to compare these archimedean, algebraic, and
$p$-adic structures on $a_d(n)$ and to use all three to conclude information on $\kappa_d$.\\

The hard-core model is a classical lattice-gas model in statistical mechanics,
with a history going back to the study of residual entropy, including
Pauling's work on ice \cite{Pauling1935,BaxterBook1982}. It subsequently
became an important model in probability and theoretical computer science,
where questions concerning Gibbs measures on independent sets and uniqueness
have played a central role \cite{Weitz2006,Sly2010}. Related hard-constraint
models have also provided important exactly solvable examples
\cite{BrightwellWinkler2002}. 
The hard-core model also has a natural connection to percolation theory through
the study of Gibbs-measure uniqueness and phase transitions. In particular,
disagreement percolation relates the propagation of boundary conditions in a
Gibbs measure to the percolation of disagreement clusters
\cite{vanDenBergMaes1994,vanDenBergSteif1994, HuChen1991, LiuEvans2000}. In fact, Deninger in~\cite{deninger2010p} studied such entropy constants using $p$-adic tools (but from a different perspective to ours) and developed the notion of $p$-adic entropy. For the hard-core model on
bipartite graphs, this connection is especially direct: phase coexistence can
be characterized in terms of percolation of disagreements, and the critical
activity can be bounded in terms of the corresponding site-percolation
threshold \cite{vanDenBergSteif1994,GeorgiiHaggstromMaes2001}. \\

In dimension $d=1$, the model is completely solvable. The graph $T_1(n)$ is
the cycle graph, and so $a_1(n)$ is the $n$-th Lucas number, giving us $\kappa_1=\frac{1+\sqrt5}{2}$ (see Example~\ref{ex:lucas} for more details). From the transfer-matrix perspective, the
reason for this algebraicity is particularly transparent: a single finite
matrix governs the sequence $a_1(n)$ as $n$ varies, and $\kappa_1$ is its
Perron eigenvalue. More generally, the entropy of a one-dimensional shift of
finite type is governed by a fixed finite transfer matrix and is the
logarithm of a weak Perron number; in the mixing case, the corresponding
number is a Perron number \cite{Lind1984}.\\

The first genuinely nontrivial higher-dimensional case is $d=2$, the
classical hard-square lattice gas. The constant $\kappa_2$ has been studied
extensively using corner-transfer-matrix methods, rigorous bounds, numerical
computations, and strip approximations
\cite{baxter1980hard,CalkinWilf1998,baxter1999planar,Pavlov2012}.
Nevertheless, no exact expression for $\kappa_2$ is known, and its
algebraicity or transcendence remains open
\cite{liang2025independent}. The corresponding question for the topological
entropy $\log\kappa_2$ has likewise been recorded as open
\cite{ban2021topological}. This stands in contrast with exactly solvable
models such as the hard-hexagon model, for which Baxter's solution gives an
algebraic entropy constant \cite{Baxter1980}.\\

The transfer-matrix formulation also explains why the one-dimensional
argument does not immediately extend. For $d=1$, a fixed finite matrix
controls the entire sequence $a_1(n)$, whereas for $d\geq2$ the relevant
transfer matrix grows with the transverse system size. Thus the mechanism
that makes $\kappa_1$ algebraic has no direct higher-dimensional analogue.
This is consistent with the general theory of multidimensional shifts of
finite type: Hochman and Meyerovitch proved that the entropies of
$\mathbb Z^d$ shifts of finite type, for $d\geq2$, are precisely the
nonnegative right recursively enumerable numbers \cite{HM2010}. In
particular, the general theory does not determine the arithmetic nature of
any particular $\kappa_d$.\\

Our approach instead exploits arithmetic information contained in the
finite-volume counts themselves. Gauss-Dold congruences arise naturally in
the study of periodic-point counts and traces of powers of integer matrices;
in particular, traces of powers of an integer matrix satisfy the relevant
congruences \cite{Arnold2006,Zarelua2008}. In the hard-core setting, however,
the transfer matrix changes with the system size, so the usual fixed-matrix
argument does not directly apply. We obtain analogous congruences instead
from the translation action on independent sets of the torus. This produces
strong prime-power divisibility properties of $a_d(n)$ and provides the
starting point for our arithmetic investigation of $\kappa_d$.\\

Our first results establish the basic asymptotic framework in which the
arithmetic questions are posed. We prove effective convergence estimates for
the finite-volume entropy and show that $\{\kappa_d\}_{d=1}^\infty$ is decreasing
and converges to $\sqrt2$ as $d\to\infty$ with a convergence rate of $O(\frac{1}{d})$ (with explicit bounds); see
Theorems~\ref{thm:existence} and~\ref{thm:decreasing}. The transfer-matrix
construction gives, for each finite transverse size, an algebraic spectral
radius whose normalized powers converge to $\kappa_d$
(Proposition~\ref{prop:Perron_Frob}). These results provide the
archimedean and algebraic approximations to $\kappa_d$ that is compared
with the arithmetic information below.\\

The main combinatorial result that allows us to turn this combinatorial problem into an algebraic one is a prime power congruence. Namely, for every $d\geq1$, every prime $p$, and every $k\geq2$, we have that $a_d(p^k)\equiv a_d(p^{k-1})\pmod{p^k}$. 
Equivalently, $\nu_p\!\left(a_d(p^k)-a_d(p^{k-1})\right)\geq k$. This is Theorem~\ref{prop:Gauss}. Its proof is purely combinatorial, as it
comes from the translation action on independent sets of $T_d(p^k)$ and
does not use the transfer-matrix formalism. In particular, the congruence
gives, for every prime $p$, a canonical $p$-adic limit $\kappa_d^{(p)}\in1+p\mathbb Z_p$ that satisfies $a_d(p^k)\equiv\kappa_d^{(p)}\pmod{p^k}$. Thus the finite-volume sequence gives rise to two a priori different
limits: the real entropy constant $\kappa_d$ and, for every prime $p$, a
$p$-adic companion $\kappa_d^{(p)}$.\\

The congruence also leads to a useful notion of a tame prime, for
which the first nontrivial $p$-adic information is sufficiently rigid to
control possible coincidences with $\kappa_d$. We show that $2$ and $3$ are
tame for every $d$ (Theorem~\ref{thm:p2}), while $a_d(5)\equiv1+5(2^{d-1}+4^{d-1})\pmod{25}$, so that $5$ is tame whenever $d\not\equiv3\pmod4$
(Proposition~\ref{thm:p5}).\\

% The main arithmetic consequences arise from comparing these congruences with
% the hypothetical algebraicity of $\kappa_d$. We first obtain an
% approximation theorem (Theorem~\ref{thm:master}) which shows that, if
% $\kappa_d$ is algebraic, then the finite-volume quantities
% $a_d(n)^{1/n^d}$ cannot approach $\kappa_d$ arbitrarily rapidly, and that the algebraic degree of $a_d(n)^{1/n^d}$ must be unbounded, unless an
% exact coincidence $a_d(n)=\kappa_d^{n^d}$ occurs. \\

Our second main ingredient is an approximation theorem. Assuming that
$\kappa_d$ is algebraic, Theorem~\ref{thm:master} gives a quantitative
restriction on how rapidly the finite-volume approximations $a_d(n)^{1/n^d}$ can converge to $\kappa_d$ and a quantitative lower bound on the algebraic degree of $a_d(n)^{1/n^d}$. More precisely, the theorem shows that an
algebraic $\kappa_d$ cannot be approximated by these quantities beyond the
rate permitted by its degree and arithmetic complexity, unless the
approximation is exact. Thus, if $a_d(n)\neq \kappa_d^{n^d}$ for every $n \geq 1$, then sufficiently rapid convergence of $a_d(n)^{1/n^d}$ to $\kappa_d$
forces $\kappa_d$ to be transcendental. This provides the bridge between
the asymptotic behavior of the finite-volume model and the arithmetic
question of the nature of $\kappa_d$.\\

This motivates the study of the exceptional set $E_d=\{n\geq1:a_d(n)=\kappa_d^{n^d}\}$. Here the $p$-adic congruences impose surprisingly strong global restrictions.
If $\kappa_d$ is algebraic and $E_d\neq\emptyset$, then
Theorem~\ref{thm:coprimality} shows that $\gcd(E_d)>1$, and thus all exceptional indices have a common prime divisor. Moreover, if $p$
is tame and $m=\nu_p(a_d(p)-1)$, then
Theorem~\ref{thm:finiteexp} shows that only finitely many powers of $p$ can
belong to $E_d$: at most $m$ for odd $p$, and at most one when $p=2$. Prime-power coincidences also force strong restrictions on the algebraic
degree of $\kappa_d$. If $\kappa_d$ is algebraic and $p^k\in E_d$, then
Theorem~\ref{thm:radical-degree} implies that $[\mathbb Q(\kappa_d):\mathbb Q]$ is a power of $p$, and for tame odd primes $p$ one obtains the more precise
restriction, which is that $[\mathbb Q(\kappa_d):\mathbb Q] = p^r$ where $r$ satisfies $r\leq kd\leq r+\nu_p(a_d(p)-1)-1$. Consequently, if the degree of $\kappa_d$ is not a prime power, then no
prime power can belong to $E_d$.\\

Finally, we obtain a local rigidity statement connecting the hypothetical
algebraic number $\kappa_d$ directly to its $p$-adic companion $\kappa_d^{(p)}$. If
$\kappa_d$ is algebraic and $p^k\in E_d$, then
Theorem~\ref{thm:localrigidity} implies that for every prime
$\mathfrak p$ of $\mathbb Q(\kappa_d)$ lying above $p$, we have that $\kappa_d\equiv1\pmod{\mathfrak p}$. Thus an exact finite-volume coincidence forces $\kappa_d$ to satisfy a
simultaneous local condition at every place above the witnessing prime.
As an application, we obtain explicit bounds on a witnessing prime in
special algebraic families, including Pisot and Salem values
(Corollary~\ref{thm:normbound}).\\

Taken together, these results give a hierarchy of obstructions to the
algebraicity of $\kappa_d$. The prime-power congruence produces canonical
$p$-adic limits; the approximation theorem constrains how an algebraic
$\kappa_d$ can be approached by finite-volume quantities; and any exact
coincidence is subject to global, degree-theoretic, and local restrictions.
We do not determine the algebraic or transcendental nature of $\kappa_d$ for
$d\geq2$. Rather, our results show that any hypothetical algebraic
$\kappa_d$ would have to reconcile these archimedean, $p$-adic, and
algebraic-number-theoretic constraints simultaneously. We summarize the main algebraicity results of $\kappa_d$:

\begin{theorem}
Let $d\ge2$ and suppose $\kappa_d$ is algebraic. Then: 

\begin{itemize}
    \item The sequence $\{-\frac{1}{n^d} \log |a_d(n)-\kappa_d^{n^d}|\}_{n=1}^\infty$ is bounded.
    \item $\deg(a_d(n)^{\frac{1}{n^d}}) \to \infty$. 
    \item $\kappa_d$ is not a quadratic unit. 
    \item For every $p \in \Tame_d$ we have that $\kappa_d^{(p)}\ne\kappa_d^{p^{kd}}$ for every large $k$ and $\kappa_d^{(p)}$ is never a $\mathfrak p$-adic limit point of the sequence $\big\{\kappa_d^{p^{kd}}\big\}_{k=1}^\infty$ for every $\mathfrak{p} \mid p$.
\end{itemize}

If, in addition, there exists some prime $p$ such that $p^k\in E_d$ for some $k$, then
\begin{itemize}
\item The intersection $\{p^j \colon j \in\mathbb{N}\} \cap E_d$ is finite if $p \in \Tame_d$. 
\item $p\le|N_{K/\Q}(\kappa_d-1)|$.
\item $[\Q(\kappa_d):\Q]$ is a power of $p$.
\item if $f$ is the minimal polynomial of $\kappa_d$, then $f(x)\equiv(x-1)^{\deg f}\pmod p$.
\item $\kappa_d$ is $\mathfrak p$-integral at every prime $\mathfrak p$ of $K$ above $p$ with $\kappa_d\equiv1\pmod{\mathfrak p}$. 
\item For $p=2$ or $p=3$, we have that $[\Q(\kappa_d):\Q]=p^{kd}$. 
\end{itemize}
\end{theorem}

\medskip
\noindent\textbf{Organization.} In Section~\ref{sec:kappa} we construct $\kappa_d$ and the transfer-matrix formalism. In Section~\ref{sec:tame} we prove the congruence of $a_d(p^k)$ and introduce the notion of tame primes, including the circulant-graph computation at $p=2,3,5$, and a discussion on the $d=2$ case. In Section~\ref{sec:padic} we construct $\kappa_d^{(p)}$ and we prove our approximation theorem of $\kappa_d$, together with a genuinely local comparison theorem at witnessing primes. In Section~\ref{sec:Ed} we study the set $E_d$ by looking at global rigidity, finiteness along a fixed prime, and explicit prime bounds for quadratic units, Pisot, and Salem values of $\kappa_d$.

\medskip
\noindent\textbf{Acknowledgments.} We thank the Department of Mathematics at the Weizmann Institute of Science (and especially Francesco M. Saettone) for their hospitality while working on this project. We wish to thank Itai Benjamini, Zhenhao Cai, Gady Kozma, and Eitan Sayag for productive mathematical discussion, and we wish to thank Sa'ar Zehavi for valuable discussions on the different algebraic number theory aspects of this project. We wish to thank Yotam Hendel, Edan Rein,  Francesco M. Saettone, and Sa'ar Zehavi for their input and comments on the different versions of this paper. 

\subsection*{Notation}\label{sec:notation}

We collect notation used throughout. For background on $p$-adic analysis we refer to \cite{Gouvea,Serre}, and for algebraic number theory to \cite{Marcus,Neukirch}.

\begin{description}[leftmargin=2.2cm,style=nextline]
\item[$\Z,\Q,\R,\C$] the integers, rationals, reals, complex numbers. We denote divisibility of integers using $\mid$. 
\item[$\quot{\Z}{n\Z}$] the integers modulo $n$, we write all quotient groups and rings this way (slanted fraction) rather than stacked, for typographic consistency.
\item[$\F_q$] the field with $q$ elements, $\F=\F_p$ when $p$ is fixed by context.
\item[$\Zp,\Qp,\Cp$] the $p$-adic integers, $p$-adic numbers, and (metric completion of an algebraic closure of $\Qp$).
\item[$\nu_p$] the $p$-adic valuation on $\Z$, $\Q$, or $\Qp$, normalized by $\nu_p(p)=1$. For a finite extension $L/\Qp$ we write $v_L$ (or $v_\mathfrak p$, below) for the valuation on $L$ normalized the same way on $\Qp\subset L$, i.e.  $v_L(p)=e_L$ and $v_L|_{\Qp} = e_Lv_p$. We denote by $\log_p(\cdot)$ the $p$-adic logarithm function (as opposed to $\log$, which denotes the real logarithm function. 
\item[$\mathfrak p\mid p$] a prime ideal $\mathfrak p$ of the ring of integers $\mathcal O_K$ of a number field $K$ lying above the rational prime $p$ (i.e.\ $\mathfrak p\cap\Z=p\Z$), $e_\mathfrak p,f_\mathfrak p$ denote its ramification index and residue degree. We have the equality $\sum_{\mathfrak p\mid p}e_\mathfrak p f_\mathfrak p=[K:\Q]$.
\item[$\omega(u),\ \langle u\rangle$] for a unit $u$ in the ring of integers of a local field with finite residue field $\F_q$, the Teichm\"uller lift $\omega(u)$ is the unique $(q-1)$-th root of unity congruent to $u$ modulo the maximal ideal, and $\langle u\rangle=\frac{u}{\omega(u)}$ is its principal unit part, so $u=\omega(u)\langle u\rangle$. For more information, see~\cite[Ch.~II]{Serre}.
\item[$\mu$] the number-theoretic M\"obius function.
\item[$\Cay(G,S)$] the Cayley graph of a group $G$ with respect to some generating set $S$ (that satisfies $1_G \notin S=S^{-1}$): vertex set $G$, with $x\sim y$ iff $xy^{-1}\in S$.
\item[$\Ind(\Gamma)$] the set of independent sets of a graph $\Gamma$ with $i(\Gamma)=|\Ind(\Gamma)|$.
\item[$\Stab_G(x)$] the stabilizer of $x$ under a group action of $G$, and with $\Fix_H(X)$ being the fixed points of a subgroup $H$ acting on a set $X$.
\item[$\tr,\ \rho(A)$] the trace of a matrix, and the spectral radius of a matrix $A$ (largest modulus of an eigenvalue).
\item[$\Mat_r(\Z)$] $r\times r$ matrices with integer entries.
\item[$\Gal(L/\Q)$] the Galois group of a Galois extension $L/\Q$.
\item[{$K$}] a number field, with ring of integers $\mathcal O_K$, and its degree $[K:\Q]$, $N_{K/\Q},\ \operatorname{Tr}_{K/\Q}$ the norm and trace to $\Q$, $\operatorname{disc}(K)$ its discriminant. The \textbf{conjugates} of $\alpha\in K$ are the images of $\alpha$ under the $[K:\Q]$ embeddings $K\hookrightarrow\C$.
\end{description}

\section{The Constants $\kappa_d$ and the Transfer Matrices $M_d(n)$}\label{sec:kappa}

We start by defining the constants $\kappa_d$ and proving some basic properties that we use throughout this paper. 

\begin{definition}\label{def:torus}
For $d,n\ge1$, the \textbf{$d$-dimensional discrete torus of side $n$} is
\[
T_d(n)\ =\ \Cay\big((\quot{\Z}{n\Z})^d,\,S_n\big),\qquad S_n=\{\pm e_1,\dots,\pm e_d\}\setminus\{0\}\subseteq(\quot{\Z}{n\Z})^d,
\]
where $e_i$ is the $i$-th standard basis vector. We write $a_d(n)=i(T_d(n))$ and we define $\kappa_d=\lim_{n\to\infty}a_d(n)^{1/n^d}$. By convention, we set both $T_d(1)$ and $T_0(n)$ to the single-vertex graphs with no edges, and we set $a_0(n)=2$ for every $n$.
\end{definition}

\begin{remark}
    For every $n\ge3$, the graph $T_d(n)$ is simple and $2d$-regular. At $n=2$, since $e_i\equiv-e_i$, we have that $S_2=\{e_1,\dots,e_d\}$ and therefore $T_d(2)$ is the $d$-dimensional hypercube graph which is $d$-regular. In the $d=1$ case we get the cyclic graph on $n$ vertices, which is $2$-regular for every $n \geq 3$, $1$-regular if $n=2$, and $T_1(1)$ contains no edges. The study of independent sets of the hypercube graph goes back to the asymptotic computation of Korshunov and Sapozhenko~\cite{korshunov1983number, sapozhenko1987number}, which was later expanded upon by Galvin~\cite{galvin2011threshold, galvin2019independent}. For more information, see~\cite{JenssenPerkins2020}.
\end{remark}

\begin{example}\label{ex:lucas}
For $d=1$ and $n \geq 2$, we have that $a_1(n)=L_n$, the $n$-th Lucas number, where the Lucas numbers are defined by the recurrence $L_{n+2}=L_{n+1}+L_n$ with $L_0=2$ and $L_1=1$. The Lucas numbers can be expressed using the formula $L(n) = (\frac{1+\sqrt{5}}{2})^{n}+(\frac{1-\sqrt{5}}{2})^n$, and so $\kappa_1=\frac{1+\sqrt{5}}{2}$.
\end{example}

We now establish the existence of $\kappa_d$, together with an explicit two-sided envelope and an effective convergence rate. In order to do so, we need to introduce an auxiliary graph that we use:

\begin{definition}
    Let $B_d(n)$ be the \textbf{grid graph} on $\{0,\dots,n-1\}^d$ (the $d$-fold Cartesian product of the path on $n$ vertices, with no wraparound), and denote $b_d(n)=i(B_d(n))$.
\end{definition}

\begin{remark}
    In fact, when defining the hard square entropy constant $\kappa_2$, one traditionally uses the grid graph $B_2(n)$ instead of the torus graph $T_2(n)$ (for example, see~\cite{BaxterBook1982}). Yet,  we use the definition based upon the torus in order to define $\kappa_d$ since, as we see in Section~\ref{sec:tame}, the finite group $(\sfrac{\Z}{n\Z})^d$ acts on the torus $T_d(n)$ and this action allows us to translate the problem from a combinatorial one to an arithmetic one.   
\end{remark}

 Note that $T_d(n)$ is exactly $B_d(n)$ together with one bundle of $n^{d-1}$ wraparound edges per axis, so $T_d(n)$ has $dn^{d-1}$ more edges than $B_d(n)$.

\begin{theorem}\label{thm:existence}
For every $d\ge1$ the constant $\kappa_d$ is well defined and for every $n\ge1$ we have that 

\begin{enumerate}
    \item $\kappa_d^{\,n^d}\,2^{-dn^{d-1}}\ \le\ a_d(n)\ \le\ \kappa_d^{\,(n+1)^d}$, 
    \item $
\left|\frac1{n^d}\log a_d(n)-\log\kappa_d\right|\ \le\ \frac{d\log2}n.$
\end{enumerate}
\end{theorem}
\begin{proof}
As the proof of this theorem is long, we break it up into steps:\\

\underline{Step 1: edge deletion.} For any graph $G$ and a single edge $uv$, deleting it to form $H=G-uv$ gives that $i(G)\le i(H)\le2i(G)$. This is true since every $G$-independent set is $H$-independent, and every $H$-independent set either avoids one of $u,v$ (and therefore is already $G$-independent) or contains both, in which case removing $v$ leaves a $G$-independent set, and this is injective on such sets, so $i(H)\le i(G)+i(G)$. Iterating over $m$ deleted edges gives $i(G)\le i(H)\le2^mi(G)$. Applied to $G =T_d(n)$ and $H=B_d(n)$, in which $dn^{d-1}$ edges were deleted, we get that:
\begin{equation*}\label{eq:torusgridcompare}
a_d(n)\ \le\ b_d(n)\ \le\ 2^{dn^{d-1}}a_d(n).
\end{equation*}
In particular, we get that $a_d(n)^{\frac{1}{n^d}} \leq b_d(n)^{\frac{1}{n^d}} \leq 2^{\frac{d}{n}}a_d(n)^{\frac{1}{n^d}}$.\\

\underline{Step 2: monotonicity of $b_d$.} For every graph $G$ and every vertex $v$ of $G$, partitioning the independent sets of $G$ by whether they contain $v$ gives the formula $i(G)=i(G-v)+i(G-N[v])\ge i(G-v)$, where $N[v]$ is the closed neighborhood of $v$. Iterating over several deleted vertices shows that deleting vertices from a graph (i.e.\ passing to an induced subgraph) never increases the number of independent sets. Therefore, for $m\le n$, the sub-box $\{0,\dots,m-1\}^d\subseteq\{0,\dots,n-1\}^d$ induces, inside $B_d(n)$, exactly the graph $B_d(m)$ (as the grid-edge relation is local and does not see the ambient box size), so $B_d(m)$ is an induced subgraph of $B_d(n)$ and we get that $b_d(m)\ \le\ b_d(n)$ for every $2 \leq m \leq n$. \\

\underline{Step 3: submultiplicativity along multiples.} Partitioning $\{0,\dots,kn-1\}^d$ into $k^d$ translated copies of $\{0,\dots,n-1\}^d$ and restricting an independent set of the big box to each copy gives an injection $\Ind(B_d(kn))\hookrightarrow\Ind(B_d(n))^{k^d}$ (a $B_d(kn)$-independent set restricted to any sub-box is $B_d(n)$-independent, and the restrictions to the $k^d$ sub-boxes together determine the original set), so $b_d(kn)\ \le\ b_d(n)^{k^d}$.\\

\underline{Step 4: super-multiplicativity via buffering.} Inside $\{0,\dots,k(n+1)-2\}^d$, insert a single empty coordinate layer between each pair of axis-adjacent $n$-blocks (i.e., $k-1$ buffer layers along each axis, so the total side length is $kn+(k-1)=k(n+1)-1$), leaving every buffer vertex unoccupied and placing an arbitrary $B_d(n)$-independent set independently in each of the $k^d$ blocks produces a $B_d\big(k(n+1)-1\big)$-independent set, injectively in the $k^d$-tuple of choices (distinct blocks are separated by an empty layer along every axis, so no adjacency is ever created between blocks). Thus we get that $b_d(n)^{k^d}\ \le\ b_d\big(k(n+1)-1\big)$.\\

\underline{Step 5: existence of the limit.} In order to show that $\kappa_d$ is well defined, we set it to be $\kappa_d=\inf_{n\ge1}b_d(n)^{1/n^d}\in(0,2]$ (which is finite and positive since $1\le b_d(n)\le 2^{n^d}$ term-wise with respect to $n$) and we show that it equals the limit in  Definition~\ref{def:torus}. By the definition of the infimum we get that $b_d(n)\ \ge\ \kappa_d^{\,n^d}$ for every $n \geq 1$, and so $\liminf_n b_d(n)^{1/n^d}\ge\kappa_d$. For the reverse inequality, fix $m\ge2$ and for every $n\ge m$, by the Euclidean algorithm we can write $n=q_n m+r_n$ where $q_n\ge1$ and $0\le r_n<m$, and so $n\le(q_n+1)m$. By the monotonicity of $\{b_d(n)\}_{n=1}^\infty$ and from the fact that $b_d(kn)\ \le\ b_d(n)^{k^d}$, we get that 
\[
b_d(n)\ \le\ b_d\big((q_n+1)m\big)\ \le\ b_d(m)^{(q_n+1)^d}.
\]
By taking $n^d$-th roots we get that $b_d(n)^{1/n^d}\le \big[b_d(m)^{1/m^d}\big]^{(\frac{(q_n+1)m}{n})^d}$. As $n\to\infty$ with $m$ fixed we have that $q_n\to\infty$ and $1\le(q_n+1)\frac{m}{n}<(q_n+1)\frac{m}{qm}=1+\frac{1}{q_n}\to1$, so $(q_n+1)\frac{m}{n}\to1$. By continuity of the function $x\mapsto x^\alpha$ at $\alpha=1$ (with base $b_d(m)^{\frac{1}{m^d}}>0$ fixed) we get that $\limsup_n b_d(n)^{\frac{1}{n^d}}\le b_d(m)^{\frac{1}{m^d}}$. Since $2\le m$ was arbitrary, we have that $\limsup_n b_d(n)^{\frac{1}{n^d}}\le\inf_{m\ge2}b_d(m)^{\frac{1}{m^d}}=\kappa_d$. Note that the infimum over $m\ge2$ agrees with that over all $m\ge1$, since we have that $b_d(m) \leq 2^{m^d}$ and so $b_d(m)^{\frac{1}{m^d}} \leq 2 =b_d(1)$. Combined with the liminf bound, we get that $b_d(n)^{\frac{1}{n^d}}\to\kappa_d$ as $n\to\infty$.\\ 

\underline{Step 6: the limit of $b_d(n)^{\frac{1}{n^d}}$.} Since $b_d(n)^{k^d}\ \le\ b_d\big(k(n+1)-1\big)$, then by sending $k \to \infty$ and writing $N=k(n+1)-1$, we get that
\[
b_d(n)\ \le\ b_d(N)^{\frac{1}{k^d}}\ =\ \Big[b_d(N)^{\frac{1}{N^d}}\Big]^{\frac{N^d}{k^d}}.
\]
Together with Step 5, we can conclude that $b_d(N)^{\frac{1}{N^d}}\to\kappa_d$ as $k\to\infty$ (equivalently $N\to\infty$), and $\frac{N^d}{k^d}=\frac{\big(k(n+1)-1\big)^d}{k^d}\to(n+1)^d$, since $b_d(n)$ is a fixed quantity bounded above by a sequence converging to $\kappa_d^{(n+1)^d}$, we conclude $b_d(n)\le\kappa_d^{(n+1)^d}$. Together with the fact that $a_d(n)\ \le\ b_d(n)\ \le\ 2^{dn^{d-1}}a_d(n)$ and with the fact that $b_d(n)\ \ge\ \kappa_d^{\,n^d}$, we get that $a_d(n)\le b_d(n)\le\kappa_d^{(n+1)^d}$ and $a_d(n)\ge b_d(n)2^{-dn^{d-1}}\ge\kappa_d^{n^d}2^{-dn^{d-1}}$, which is exactly item 1. Taking logarithms and dividing by $n^d$ also shows that $a_d(n)^{1/n^d}\to\kappa_d$. Since $a_d(n)\ \le\ b_d(n)\ \le\ 2^{dn^{d-1}}a_d(n)$, we get that $|\log a_d(n)-\log b_d(n)|\le dn^{d-1}\log2$, and so $\big|\log a_d(n)/n^d-\log b_d(n)/n^d\big|\le d\log2/n\to0$, and $\log b_d(n)/n^d\to\log\kappa_d$ by Step 5.\\

\underline{Step 7: the rate (item 2).} Fix $n\ge1$, for $k\ge1$ we can tile $\{0,\dots,kn-1\}^d$ by $k^d$ copies of $\{0,\dots,n-1\}^d$ and delete every cross-tile edge of $T_d(kn)$, including its $d$ wraparound bundles. A direct count (each of the $d$ axes contributes $k$ inter-tile boundaries, each carrying $(kn)^{d-1}$ edges) shows that $dk(kn)^{d-1}=dk^dn^{d-1}$ edges are deleted, and what remains is the disjoint union of $k^d$ copies of $B_d(n)$ (within a tile, and ignoring the edges leaving it, $T_d(kn)$ restricts to exactly the box grid $B_d(n)$). By the edge-deletion bound of Step 1 (applied with $G=T_d(kn)$, the graph with more edges, and $H$ the disjoint union), we get that
\[
0\ \le\ \log b_d(n)^{k^d}-\log a_d(kn)\ \le\ m\log2,
\]
and by dividing by $(kn)^d=k^dn^d$ and letting $k\to\infty$ (using the fact that $a_d(kn)^{1/(kn)^d}\to\kappa_d$ from Step 6, applied along the subsequence $kn$), we get the one-sided bound
\begin{equation*}
0\ \le\ \frac{\log b_d(n)}{n^d}-\log\kappa_d\ \le\ \frac{d\log2}n.
\end{equation*}
Separately, the inequality $a_d(n)\ \le\ b_d(n)\ \le\ 2^{dn^{d-1}}a_d(n)$ (which comes from deleting $dn^{d-1}$ edges) gives us that
\begin{equation*}
-\frac{d\log2}n\ \le\ \frac{\log a_d(n)}{n^d}-\frac{\log b_d(n)}{n^d}\ \le\ 0.
\end{equation*}
Therefore, we can conclude that $\left|\frac{\log a_d(n)}{n^d}-\log\kappa_d\right|\ \le\ \frac{d\log2}n$, which is exactly item 2.
\end{proof}

\begin{remark}
Item 2 in Theorem~\ref{thm:existence} is in fact stronger than the naive attempt to apply item $1$. This is true since given that $\kappa_d^{n^d}\le b_d(n)\le\kappa_d^{(n+1)^d}$, we can only conclude that $\log\kappa_d\le\frac{\log b_d(n)}{n^d}\le\big(\frac{(n+1)}{n}\big)^d\log\kappa_d$, and the gap $\big((1+\frac{1}{n})^d-1\big)\log\kappa_d$ on the right is not bounded by $d\frac{\log2}{n}$ once $d>1$ (since Bernoulli's inequality only gives us the weaker $(1+\frac{1}{n})^d-1\le\frac dn(1+\frac{1}{n})^{d-1}$).
\end{remark}

\begin{lemma}\label{lem:naive_bound}
For every $d\ge1$ and even $n$ we have that $2^{\frac{n^d}{2}}\le a_d(n)\le3^{\frac{n^d}{2}}$, and so $\sqrt2\le\kappa_d\le\sqrt3$.
\end{lemma}
\begin{proof}
For the lower bound, we can partition $T_d(n)$ into two sets, $E=\left\{x:\sum_i x_i\equiv0\pmod2\right\}$ and $O=\left\{x:\sum_i x_i\equiv1\pmod2\right\}$, each of size $\frac{n^d}{2}$. Every subset of either $E$ or $O$ is an independent set, and so $a_d(n)\ge2^{\frac{n^d}{2}}$. For the upper bound, since $n$ is even, the edges $\{x,x+e_1\}$ over all $x$ with $x_1$ even partition the vertex set into $\frac{n^d}{2}$ disjoint pairs, an independent set meets each pair in at most one vertex, giving at most $3$ choices per pair (either endpoint, or neither), and so $a_d(n)\le3^{\frac{n^d}{2}}$.%Taking $n^d$-th roots and $n\to\infty$ along even $n$ in Theorem~\ref{thm:existence} gives $\sqrt2\le\kappa_d\le\sqrt3$.
\end{proof}

\begin{remark}
The bounds presented in Lemma~\ref{lem:naive_bound} are far from sharp. For example, in the $d=2$ case, we in fact know that $\kappa_2\in[1.5030,1.5031]$ since $\kappa_2\approx1.503048$. In fact, Theorem~\ref{thm:decreasing} below together with Example~\ref{ex:lucas} gives us that $\sqrt{2} < \kappa_d \leq \kappa_1 = \frac{1+\sqrt{5}}{2}$ for every $d \geq 1$, which is already a better upper bound than Lemma~\ref{lem:naive_bound}. Yet, the bound $\sqrt2\le\kappa_d\le\sqrt3$ is used extensively, especially in Sections~\ref{sec:padic} and~\ref{sec:Ed}. In addition, there has been much research on the numerical approximation and on the computation of the digits of $\kappa_2$. For more details, see for example~\cite{baxter1999planar, CalkinWilf1998,Pavlov2012}. 
\end{remark}

The following theorem summarizes the main properties of $\{\kappa_d\}_{d=1}^\infty$ as a family.

\begin{theorem}\label{thm:decreasing}
For every $d \geq 1$ we have
\begin{enumerate}
    \item $\kappa_{d+1} \leq \kappa_d$,
    \item $\sqrt2 < \sqrt2\,
\bigl(1+2^{-2d}\bigr)^{1/[2(2d^2+1)]} \ \le\ \kappa_d\ \le\ \big(2^{2d+1}-1\big)^{1/4d}$,
    \item  $0\ \le\ \kappa_d-\sqrt2\ <\ \frac{\sqrt2\,(e-1)\log2}{4d}$.
\end{enumerate}
In particular, $\kappa_d \to \sqrt{2}$  as $d\to \infty$.  
\end{theorem}
\begin{proof}
For part 1, we can slice $T_{d+1}(n)$ along its last coordinate into $n$ layers (each isomorphic to $T_d(n)$), which restricts an independent set of $T_{d+1}(n)$ to an independent set on each layer. This gives us the inequality $a_{d+1}(n)\le a_d(n)^n$, and by taking $n^{d+1}$-th roots and $n\to\infty$, we get that $\kappa_{d+1}\le\kappa_d$.\\

For part 2, note that for a $2d$-regular graph $G$ on $N$ vertices, the Kahn-Zhao theorem (see ~\cite{Kahn2001,Zhao2010}) tells us that $i(G)\le\big(2^{2d+1}-1\big)^{N/4d}$. By applying it to the graph $T_d(n)$ (which is $2d$-regular on $n^d$ vertices for $n\ge3$) we get that $a_d(n)\le(2^{2d+1}-1)^{n^d/4d}$, and so we get that $\kappa_d\le(2^{2d+1}-1)^{1/4d}$. For the lower bound, let $n\ge6$ be even and put $N=n^d$. As in Lemma~\ref{lem:naive_bound}, the torus $T_d(n)$ is bipartite and can be partitioned into two sets,  $E=\left\{x:\sum_i x_i\equiv0\pmod2\right\}$ and $O=\left\{x:\sum_i x_i\equiv1\pmod2\right\}$,  each of cardinality $N/2$. Construct a conflict graph $\Gamma$ on $O$: two distinct vertices $u,v\in O$ are adjacent in $\Gamma$ when they have a common neighbor in $E$. For fixed $u$, the possible differences $v-u$ are $\pm2e_i$ or $\pm e_i\pm e_j$ for $i\ne j$. There are at most $2d+4\binom d2=2d^2$ such vertices. Thus, the largest degree in $\Gamma$ is at most $2d^2$.\\

Since every graph of maximum degree $\Delta$ has an independent set of size at least $|V|/(\Delta+1)$, we can find some independent set $U\subseteq O$ in $\Gamma$ with such that $|U|\ge\frac{N}{2(2d^2+1)}$. By construction, the neighbor sets $N(u)\subseteq E$ for $u\in U$ are pairwise disjoint and each has size $2d$. In addition, for every subset $S\subseteq U$, we can choose an arbitrary subset of $E\setminus N(S)$. The resulting union is independent in $T_d(n)$. Therefore we have that 
\[
a_d(n) \ge
\sum_{S\subseteq U}2^{|E|-|N(S)|} =
\sum_{S\subseteq U}2^{N/2-2d|S|} =
2^{N/2}\bigl(1+2^{-2d}\bigr)^{|U|},
\]
and thus by taking $N$-th roots and then letting $n\to\infty$ through even integers the result follows.\\

For part 3, since $2^{2d+1}-1<2^{2d+1}$ we can conclude that $(2^{2d+1}-1)^{1/4d}<2^{(2d+1)/4d}=\sqrt2\cdot2^{1/4d}$. By convexity of $t\mapsto e^t$, the graph of $e^t$ lies below the line connecting the points $(0,e^0)=(0,1)$ and $(1,e^1)=(1,e)$ for $t\in[0,1]$, i.e.\ $e^t\le1+(e-1)t$ for $t\in[0,1]$. By applying this at $t=\tfrac1{4d} \log2\in(0,\log2]\subset[0,1]$ (since $\log2<1$), we get that  $2^{1/4d}\le1+\tfrac{(e-1)\log2}{4d}$, and so (together with Lemma~\ref{lem:naive_bound}) we get that 
\[
\sqrt{2} \leq \kappa_d \leq (2^{2d+1}-1)^{1/4d}< \sqrt2\cdot2^{1/4d} \ <\ \sqrt2\Big(1+\tfrac{(e-1)\log2}{4d}\Big)\ =\ \sqrt2+\frac{\sqrt2\,(e-1)\log2}{4d},
\]
and the result follows.  %\textbf{Limits.} By item 2, $\kappa_d-\sqrt2\to0$, i.e.\ $\kappa_d\to\sqrt2$, since both $\kappa_d,\kappa_{d+1}\to\sqrt2\ne0$, their ratio $\kappa_{d+1}/\kappa_d\to1$.
\end{proof}

\begin{remark}
    Intuitively, we can view the fact that $\kappa_d \to \sqrt{2}$  as telling us that for large enough $d$, "almost all" the independent sets of $T_d(n)$ are going to be the subsets of $O$ and $E$, as in the proof of part 2 of Theorem~\ref{thm:decreasing} and as in Lemma~\ref{lem:naive_bound}. This intuition is in fact related to the concept of the limit of the independence entropy in $\mathbb{Z}^d$~\cite{louidor2013independence, korshunov1983number, ordentlich2004independent}. 
\end{remark}

We end this section with a discussion on an alternative approach to the definition of $\kappa_d$ that relies on a special matrix we call the transfer matrix. 

\begin{definition}\label{def:transfer}
For $d\ge1$, the \textbf{transfer matrix} of $a_d(n)$, denoted $M_d(n)$, is the $a_{d-1}(n)\times a_{d-1}(n)$ matrix indexed by $\Ind(T_{d-1}(n))$, with
\[
[M_d(n)]_{I,J}\ =\ \begin{cases}1&\text{if }I\cap J=\emptyset,\\0&\text{otherwise,}\end{cases}\qquad I,J\in\Ind(T_{d-1}(n)).
\]
\end{definition}

\begin{proposition}\label{prop:trace}
For every $d\ge1$ and every $n\ge2$, we have that  $a_d(n)=\tr\big(M_d(n)^n\big)$.
\end{proposition}
\begin{proof}
By slicing $T_d(n)$ into $n$ layers $L_0,\dots,L_{n-1}$ along its last coordinate, each a copy of $T_{d-1}(n)$ (indexed cyclically, since the last coordinate is also taken mod $n$), we get that an independent set $S$ of $T_d(n)$ corresponds, via $I_j=S\cap L_j$ (identified with a subset of the common vertex set of $T_{d-1}(n)$), to a sequence $(I_0,\dots,I_{n-1})$ with each $I_j$ independent in $T_{d-1}(n)$ (as edges within a layer are edges of $T_d(n)$) and, for every $j$ (indices mod $n$), $I_j\cap I_{j+1}=\emptyset$. This last condition is exactly the requirement that no vertex belongs to $S$ in two consecutive layers at the same position, which is necessary and sufficient for $S$ to avoid the last-coordinate edges, since layers two or more apart share no edges of $T_d(n)$ at all. This correspondence is a bijection, and a sequence $(I_0,\dots,I_{n-1})$ with $I_j\cap I_{j+1}=\emptyset$ cyclically is exactly a closed walk of length $n$ in the graph with adjacency matrix $M_d(n)$, the number of closed walks of length $n$ from a fixed vertex, summed over all vertices, is $\tr(M_d(n)^n)$.
\end{proof}

\begin{remark}
    Note that in the $d=1$ case we have that, using $a_0(n)=2$, that $M_1(n)$ is the $2\times2$ matrix indexed by $\{\emptyset,\{*\}\}$ with a single forbidden pair, i.e.\ $M_1(n)=\begin{bsmallmatrix}1&1\\1&0\end{bsmallmatrix}$). The eigenvalues of $M_1(n)$ are $\frac{1\pm\sqrt{5}}{2}$ and so $a_1(n) = (\frac{1+\sqrt{5}}{2})^{n}+(\frac{1-\sqrt{5}}{2})^n$, which recovers the result from Example~\ref{ex:lucas}.
\end{remark}

We record one further consequence of Proposition~\ref{prop:trace}, which relates the definition of $\kappa_d$ to a sequence of algebraic numbers. Specifically, since $M_d(n)$ is a real symmetric $0/1$ matrix, its spectral radius governs the values of $a_d(n)$ directly. We also isolate the comparison of transfer matrices at consecutive values of $n$, which both repairs and strengthens the naive trace-power comparison one might otherwise attempt.

% \begin{lemma}[Spectral monotonicity of the transfer matrices]\label{lem:submatrix}
% For every $d\ge2$ and $n\ge2$, $M_d(n)$ is (via the identification below) a principal submatrix of $M_d(n+1)$, so $\lambda_d(n)=\rho(M_d(n))$ satisfies .
% \end{lemma}
% \begin{proof}

% \end{proof}

\begin{proposition}\label{prop:Perron_Frob}
Let $n,d \in \mathbb{N}$ such that  $d\ge2$. Then:
\begin{enumerate}
    \item $M_d(n)$ has a real eigenvalue $\lambda_d(n)\ge0$ with $|\lambda| \le \lambda_d(n)$ for every other eigenvalue $\lambda$ of $M_d(n)$ (its \textbf{spectral radius}, $\lambda_d(n)=\rho(M_d(n))$).
    \item $\lambda_d(n)$ is a \textbf{weak Perron number}: a positive algebraic integer with $|\alpha|\le \lambda_d(n)$ for every conjugate $\alpha$ of $\lambda_d(n)$.
    \item We have an inequality $\lambda_d(n)\le\lambda_d(n+1)$.
    \item $\kappa_d = \lim_{n \to \infty} \lambda_d(n)^{1/n^{d-1}}$.
\end{enumerate}
\end{proposition}
\begin{proof}
$M_d(n)$ is real, symmetric, and has nonnegative entries, so by the Perron-Frobenius theorem \cite[Ch.~8]{HornJohnson} it has a nonnegative real eigenvalue $\lambda_d(n)$ equal to its spectral radius. Its characteristic polynomial is monic with integer coefficients, so $\lambda_d(n)$ is an algebraic integer, and every conjugate is again an eigenvalue of $M_d(n)$ (a root of the same characteristic polynomial), thus of modulus $\le\lambda_d(n)$. This proves items 1 and 2.\\

For item 3, we use the following general fact. Let $A$ be a real symmetric matrix with nonnegative entries, indexed by a finite set $X$, and let $A'$ be its principal submatrix on a subset $Y\subseteq X$ (i.e.\ $A'=A|_{Y\times Y}$), then $\rho(A')\le\rho(A)$. This is true since if $v\in\R^Y$ is a Perron eigenvector of $A'$ (nonnegative, by the Perron-Frobenius theorem \cite[Ch.~8]{HornJohnson}, as $A'$ is again real, symmetric, and entry-wise nonnegative), we can normalize it to $\|v\|=1$, and extend it by zero to a vector $\tilde v\in\R^X$. Since $A$ has nonnegative entries and $\tilde v$ vanishes off $Y$, only the $Y\times Y$ block of $A$ contributes to the Rayleigh quotient $\tilde v^TA\tilde v=v^TA'v=\rho(A')$. As $\rho(A)=\max_{\|w\|=1}w^TAw$ (the Rayleigh characterization of the top eigenvalue of a symmetric matrix), $\rho(A)\ge\tilde v^TA\tilde v=\rho(A')$.\\

Now, we identify the vertex set $\{0,\dots,n-1\}^{d-1}$ of $B_{d-1}(n)$ (the $(d-1)$-dimensional grid graph of Theorem~\ref{thm:existence}) with a subset of $(\sfrac{\Z}{(n+1)\Z})^{d-1}$, the vertex set of $T_{d-1}(n+1)$. Since $T_{d-1}(n+1)$'s wraparound edges only ever involve the coordinate value $n$, which does not occur in $\{0,\dots,n-1\}^{d-1}$, the induced subgraph of $T_{d-1}(n+1)$ on this sub-box is exactly $B_{d-1}(n)$, then every $B_{d-1}(n)$-independent set is also independent in $T_{d-1}(n+1)$, i.e.\ $\Ind(B_{d-1}(n))\subseteq\Ind(T_{d-1}(n+1))$. On the other hand $T_{d-1}(n)$ (on the same vertex set, relabeled as $(\quot{\Z}{n\Z})^{d-1}$) has “at least as many edges than $B_{d-1}(n)$ (the wraparound bundles), so $\Ind(T_{d-1}(n))\subseteq\Ind(B_{d-1}(n))$. Composing, $\Ind(T_{d-1}(n))\subseteq\Ind(T_{d-1}(n+1))$ as subsets of a common ground set. Since compatibility of two independent sets ($I\cap J=\emptyset$) does not depend on which ambient graph they are considered independent in, the entries of $M_d(n+1)$ restricted to the index subset $\Ind(T_{d-1}(n))\subseteq\Ind(T_{d-1}(n+1))$ agree exactly with $M_d(n)$, that is, $M_d(n)$ is a principal submatrix of $M_d(n+1)$. The first claim now gives $\lambda_d(n)=\rho(M_d(n))\le\rho(M_d(n+1))=\lambda_d(n+1)$.\\

For item 4, since $M_d(n)$ is symmetric it is orthogonally diagonalizable with real eigenvalues $\theta_1(n)=\lambda_d(n),\theta_2(n),\dots,\theta_{a_{d-1}(n)}(n)$, all of modulus $\le\lambda_d(n)$, then for every $m\ge1$,
\begin{equation}\label{eq:traceupper}
\tr\big(M_d(n)^m\big)\ =\ \sum_i\theta_i(n)^m\ \le\ a_{d-1}(n)\,\lambda_d(n)^m,
\end{equation}
with no parity restriction on $m$ (this only uses $|\theta_i(n)|\le\lambda_d(n)$). For \textbf{even} $m$, every term $\theta_i(n)^m\ge0$, and the term $i=1$ alone already contributes $\lambda_d(n)^m$, so
\begin{equation}\label{eq:tracelower}
\tr\big(M_d(n)^m\big)\ \ge\ \lambda_d(n)^m\qquad(m\text{ even}).
\end{equation}

\textbf{Convergence along even $n$.} Let $n$ be even. Then from  Proposition~\ref{prop:trace} together with the previous computations we get that:
\[
\lambda_d(n)^n\ \le\ a_d(n)\ \le\ a_{d-1}(n)\,\lambda_d(n)^n.
\]
Taking logarithms and dividing by $n^d$:
\[
\frac{\log\lambda_d(n)}{n^{d-1}}\ \le\ \frac{\log a_d(n)}{n^d}\ \le\ \frac{\log a_{d-1}(n)}{n^d}+\frac{\log\lambda_d(n)}{n^{d-1}}.
\]
The left inequality and $a_d(n)^{\frac{1}{n^d}}\to\kappa_d$ (Theorem~\ref{thm:existence}) give $\limsup_{n\text{ even}}\log\frac{\lambda_d(n)}{n^{d-1}}\le\log\kappa_d$. The right inequality, rearranged, gives $$\frac{\log (a_d(n))}{n^d}-\frac{\log (a_{d-1}(n))}{n^d}\le\frac{\log(\lambda_d(n))}{n^{d-1}}.$$ Since $\frac{\log a_{d-1}(n)}{n^d}=\tfrac1n\cdot(\frac{\log a_{d-1}(n)}{n^{d-1}})\to0\cdot\log\kappa_{d-1}=0$ (Theorem~\ref{thm:existence} applied one dimension down), the left side tends to $\log\kappa_d-0=\log\kappa_d$, giving $\liminf_{n\text{ even}}\frac{\log\lambda_d(n)}{n^{d-1}}\ge\log\kappa_d$. Thus we get that
\begin{equation*}
\lambda_d(n)^{\frac{1}{n^{d-1}}}\ \longrightarrow\ \kappa_d\qquad\text{as }n\to\infty\text{ through even integers.}
\end{equation*}

\textbf{Convergence along odd $n$, via monotonicity.} Let $n\ge3$ be odd, so $n-1$ and $n+1$ are even. By part 3 we have  $\lambda_d(n-1)\le\lambda_d(n)\le\lambda_d(n+1)$, so
\[
\lambda_d(n-1)^{\frac{1}{n^{d-1}}}\ \le\ \lambda_d(n)^{\frac{1}{n^{d-1}}}\ \le\ \lambda_d(n+1)^{\frac{1}{n^{d-1}}}.
\]
For either sign $\varepsilon=\pm1$ we have that $\lambda_d(n+\varepsilon)^{\frac{1}{n^{d-1}}}=\Big[\lambda_d(n+\varepsilon)^{\frac{1}{(n+\varepsilon)^{d-1}}}\Big]^{\frac{(n+\varepsilon)^{d-1}}{n^{d-1}}}$, as $n\to\infty$ (for odd values of $n$), and so $n+\varepsilon\to\infty$ through even integers. Thus, from the fact that $\lambda_d(n)^{\frac{1}{n^{d-1}}}\ \to \kappa_d$ as $n\to\infty$ over even numbers $n$, the base tends to $\kappa_d>0$, while the exponent $\frac{(n+\varepsilon)^{d-1}}{n^{d-1}}\to1$, so the whole expression tends to $\kappa_d^1=\kappa_d$ by continuity. The squeeze theorem gives $\lambda_d(n)^{\frac{1}{n^{d-1}}}\to\kappa_d$ along odd $n$ as well. Combining the even and odd cases, $\lambda_d(n)^{\frac{1}{n^{d-1}}}\to\kappa_d$ as $n\to\infty$ through all integers, as desired.
\end{proof}

\begin{remark}
    For odd values of $m$, the lower bound presented in Equations~\ref{eq:traceupper} and~\ref{eq:tracelower} can genuinely fail. For example, at $d=2$ and $n=3$ we have that $M_2(3)$ has eigenvalues $\approx3.3028,-1,-1,-0.3028$, giving $\tr(M_2(3)^3)=a_2(3)=34$, while $\lambda_2(3)^3\approx36.03$, the naive attempt to bound $\lambda_d(n)^n\le a_d(n)=\tr(M_d(n)^n)$ directly, for the exact exponent $m=n$, is therefore not available in general, and we route around it below using Proposition~\ref{prop:Perron_Frob} instead.
\end{remark}

\begin{remark}\textup{
Unlike the $d=1$ case, algebraicity of $\lambda_d(n)$ for each fixed $n$ does not transfer to $\kappa_d$ for $d\ge2$, since $M_d(n)$ genuinely grows with $n$: Hochman and Meyerovitch characterize the entropies of $\Z^d$ shifts of finite type, $d\ge2$, as exactly the right recursively enumerable numbers \cite{HM2010} - a class vastly larger than the algebraic numbers - whereas entropies of $\Z$-shifts of finite type are always logarithms of weak Perron numbers \cite{Lind1984,PavlovSchraudner2015}, matching Example~\ref{ex:lucas}. Boyd's observations on the range of Mahler measure show more generally that spectral-radius-type sequences carry little algebraicity information on their own \cite{Boyd1981}.}
\end{remark}
\begin{remark}
    The limit presented in item 4 of Proposition~\ref{prop:Perron_Frob} in the case $d=2$, that is $\kappa_2=\lim_{n \to \infty} \lambda_2(n)^{\frac{1}{n}}$, is a classical result that has been used extensively in the approximation of the hard square entropy constant, for example~\cite{CalkinWilf1998}.
\end{remark}

\section{The value of $\nu_p(a_d(p^k)-1)$ and Tame Primes}\label{sec:tame}

We now turn to an arithmetic property of $a_d(n)$ with respect to its prime power values (Theorem~\ref{prop:Gauss}). This condition, which needs no hypothesis on $d$ or on the algebraicity of $\kappa_d$, is the combinatorial result that allows us to turn the  algebraicity problem from a combinatorial one to an arithmetic one. We start with two lemmata:

\begin{lemma}\label{lem:pullback}
Let $d\ge1$,  $m\ge2$, $e\ge1$, and set $n=me$. Let $\pi:(\quot{\Z}{n\Z})^d\to(\quot{\Z}{m\Z})^d$ be coordinate-wise reduction modulo $m$, and for $J\subseteq(\quot{\Z}{m\Z})^d$ let $\pi^{-1}(J)\subseteq(\quot{\Z}{n\Z})^d$ be its full preimage. Then:
\begin{enumerate}
\item $J$ is independent in $T_d(m)$ if and only if $\pi^{-1}(J)$ is independent in $T_d(n)$,
\item $J\mapsto\pi^{-1}(J)$ is a bijection from $\Ind(T_d(m))$ onto the set of independent sets of $T_d(n)$ that are invariant under translation by the subgroup $m\cdot(\quot{\Z}{n\Z})^d\le(\quot{\Z}{n\Z})^d$.
\end{enumerate}
\end{lemma}
\begin{proof}
For part 1,  if $x,y\in\pi^{-1}(J)$ are adjacent in $T_d(n)$, then $x-y\equiv\varepsilon e_i\pmod n$ for some $i$ and $\varepsilon=\pm1$. Reducing mod $m$ gives us that $\pi(x)-\pi(y)\equiv\varepsilon e_i\pmod m$, which cannot be $0$, as that would force $e_i\equiv0\pmod m$, which is  impossible since $m\ge2$. Therefore $\pi(x)\ne\pi(y)$ must be adjacent in $T_d(m)$ and must both lie in $J$, which contradicts the independence of the set $J$.  If $u,v\in J$ are adjacent, then we must have that $u-v\equiv\varepsilon e_i\pmod m$ for some $i,\varepsilon$, and by lifting $u$ to some $x\in\pi^{-1}(u)$, we have that $y=x-\varepsilon e_i$ is a genuine neighbor of $x$ in $T_d(n)$ (as $n\ge2$), and so $\pi(y)\equiv v\pmod m$. Thus we get that $y\in\pi^{-1}(J)$, then $x,y\in\pi^{-1}(J)$ are adjacent, which contradicts the independence of $\pi^{-1}(J)$.\\

For part 2, since $\pi$ is constant on cosets of its kernel, we have that $\pi^{-1}(J)$ is invariant under the action of $\ker\pi=m(\quot{\Z}{n\Z})^d$. Conversely, any independent set $I$ that is invariant under $\ker\pi$ is a union of cosets of $\ker\pi$. Since we have an isomorphism $(\quot{\Z}{n\Z})^d/\ker\pi\cong(\quot{\Z}{m\Z})^d$ via $\pi$, then $I$ equals $\pi^{-1}(\pi(I))$. So by part 1, $\pi(I)\in\Ind(T_d(m))$ exactly when $I=\pi^{-1}(\pi(I))$ is independent.
\end{proof}

\begin{remark}\label{rem:monotone}
Item 2 in Lemma~\ref{lem:pullback} gives us an injection $\Ind(T_d(m))\hookrightarrow \Ind(T_d(n))$ whenever $m\mid n$ and $m\ge2$. Thus we can conclude that $a_d(m)\le a_d(n)$ whenever $2\le m\mid n$.
\end{remark}

\begin{lemma}\label{lem:subgroup}
Let $G=(\quot{\Z}{p^k\Z})^d$ and $K=p^{k-1}G\cong(\quot{\Z}{p\Z})^d$. If $S\le G$ is a subgroup with $K\not\le S$, then $p^k\mid[G:S]$.
\end{lemma}
\begin{proof}
Choose $g\in K\setminus S$ and write $g=p^{k-1}g_0$. Since $g\ne0$, we have that $g_0$ has order exactly $p^k$ (if the order of $g_0$ divided $p^{k-1}$ then $g=p^{k-1}g_0$ would vanish). Let $C=\langle g_0\rangle\cong\quot{\Z}{p^k\Z}$, then $g=p^{k-1}g_0$ has order $\frac{p^k}{\gcd(p^{k-1},p^k)}=p$ in $C$, and generates the unique subgroup of $C$ of order $p$ (cyclic groups have a unique subgroup of each order dividing the group order). Since subgroups of a cyclic $p$-group are totally ordered by inclusion and  since $g\notin S$, then $S\cap C$ (itself a subgroup of $C$) cannot be a nontrivial subgroup of $C$, as every nontrivial subgroup of $C$ contains the minimal one $\langle g \rangle$, therefore $S\cap C=\{0\}$. As $G$ is abelian, all subgroups are normal, so by the second isomorphism theorem $[SC:S]=[C:S\cap C]=|C|=p^k$, and $[G:S]=[G:SC][SC:S]$ is a multiple of $p^k$.
\end{proof}

We are now ready to prove the main theorem of this section, which tells us that $\{a_d(p^k)\}_{k=1}^\infty$ satisfies a special divisibility condition for every prime $p$:

\begin{theorem}\label{prop:Gauss}
For every $d\ge1$ and for every prime $p$ we have that:
\begin{enumerate}
\item $a_d(p)\equiv1\pmod p$,
\item $a_d(p^k)\equiv a_d(p^{k-1})\pmod{p^k}$ for every $k \geq 2$ (equivalently, $\nu_p\big(a_d(p^k)-a_d(p^{k-1})\big)\ge k$).
\end{enumerate}
\end{theorem}
\begin{proof}
First, we start with the $k=1$ case.  Let $G=(\quot{\Z}{p\Z})^d$ act on $X=\Ind(T_d(p))$ by (coordinate-wise) translation, this is an action by graph automorphisms of $T_d(p)$, since $G$ acts on the vertex set simply transitively and preserves the generating set $S_p$ of Definition~\ref{def:torus}. A $G$-invariant subset of the vertex set is $\emptyset$ or the whole vertex set (by simple transitivity, a $G$-invariant subset is a union of $G$-orbits, and there is only the trivial orbit and the full orbit), and the whole vertex set is not independent (it contains edges, as $p\ge2$ and $d\ge1$), so $\emptyset$ is the unique $G$-fixed point of $X$. As $G$ is abelian, $\Stab_G(I)$ is constant along each orbit (conjugate stabilizers of points in the same orbit coincide when the acting group is abelian), so every non-fixed orbit has size $[G:\Stab_G(I)]$, a proper (i.e.\ nontrivial) divisor of $|G|=p^d$, thus must be a positive power of $p$. Summing orbit sizes,
\[
a_d(p)\ =\ 1\ +\ \sum_{\text{nontrivial orbits }}|\textup{Orbit}|\ \equiv\ 1\pmod p.
\]

We now turn to the $k\ge2$ case. Let $G=(\quot{\Z}{p^k\Z})^d$ act on $X=\Ind(T_d(p^k))$ by translation, and let $K=p^{k-1}G\cong(\quot{\Z}{p\Z})^d$. By Lemma~\ref{lem:pullback} with $m=p^{k-1}\ (\ge2$, as $k\ge2$), $e=p$, the map $J\mapsto\pi^{-1}(J)$ is a bijection $\Ind(T_d(p^{k-1}))\to\Fix_K(X)$, so $|\Fix_K(X)|=a_d(p^{k-1})$, here $\Fix_K(X)$ is $G$-invariant as a set, since $G$ is abelian (so $K\trianglelefteq G$, and therefore  for $x\in\Fix_K(X)$ together with $g\in G$ and with $k\in K$ we have that $k(gx)=g(kx)=gx$), and therefore decomposes into full $G$-orbits, as does its complement. If $x\notin\Fix_K(X)$, then $K\not\le\Stab_G(x)$, so by Lemma~\ref{lem:subgroup} the orbit of $x$ has size $[G:\Stab_G(x)]$ divisible by $p^k$. Thus we get that
\[
a_d(p^k)-a_d(p^{k-1})\ =\ \sum_{\text{orbits }\not\subseteq\Fix_K(X)}|\textup{Orbit}|\ \equiv\ 0\pmod{p^k},
\]
and the result follows.  
\end{proof}

% \begin{theorem}[Main congruence]
% For every $d\ge1$, every prime $p$, and every $k\ge1$, $\nu_p\big(a_d(p^k)-a_d(p^{k-1})\big)\ge k$, where for $k=1$ this is read as  (Proposition~\ref{prop:Gauss}(a)). 
% \end{theorem}
% \begin{proof}
% Immediate from Proposition~\ref{prop:Gauss}, combining (a) (the case $k=1$) and (b) (the case $k\ge2$).
% \end{proof}

\begin{corollary}\label{cor:p_mod_1}
For every prime $p$, every $k \geq 1$, and every $d$ we have that  $a_d(p^k) \equiv 1 \pmod p$.
\end{corollary}
\begin{proof}
For $k=1$ this is part 1 of Theorem~\ref{prop:Gauss}. For $k\ge2$, theorem~\ref{prop:Gauss} gives us that $a_d(p^k)\equiv a_d(p^{k-1})\pmod{p^k}$, in particular $\pmod p$ (as $k\ge1$). By induction on $k$, we have that $a_d(p^k)\equiv a_d(p^{k-1})\equiv\cdots\equiv a_d(p)\equiv1\pmod p$ and the result follows.
\end{proof}

\begin{corollary}\label{cor:zeta-primepower}
For every $d\ge1$, every prime $p$, and every $k\ge2$, we have  $p^k\mid\sum_{j \leq k}\mu(p^{k-j})a_d(p^j)$.
\end{corollary}
\begin{proof}
The divisors of $p^k$ are $1,p,\dots,p^k$, and $\mu(\frac{p^k}{p^j})$ vanishes unless $k-j\in\{0,1\}$, where it equals $1$ and $-1$ respectively, so $\sum_{j \leq k}\mu(p^{k-j})a_d(p^j)=a_d(p^k)-a_d(p^{k-1})$, and since $k\ge2$, we have that $p^{k-1}$ is itself a prime power with exponent $\ge1$, so no boundary convention is involved. Theorem~\ref{prop:Gauss} shows this difference is divisible by $p^k$.
\end{proof}

\begin{remark}\label{rem:k1exception}
The exponent restriction $k\ge2$ in Corollary~\ref{cor:zeta-primepower} is not a mere technicality. At $k=1$ we have that $\sum_{r\mid p}\mu(\frac{p}{r})a_d(r)=a_d(p)-a_d(1)=a_d(p)-2$, using the actual value $a_d(1)=2$ (as in Definition~\ref{def:torus}), and since $a_d(p)\equiv1\pmod p$ unconditionally (from the first item of Proposition~\ref{prop:Gauss}), $a_d(p)-2\equiv-1\pmod p$, which is never $0$. So $\frac{a_d(p)-2}{p}\notin\Z$ for every prime $p$ and every $d\ge1$. %integrality of $\zeta_d$'s coefficients genuinely fails already at $n=p$, and only recovers (at prime-power indices) from $k=2$ onward. This is the same single-vertex convention responsible for the $e^x$ defect in $\zeta_1$ noted above, both are traceable to the fact that the natural "fixed-point count" in the orbit-counting proof of Proposition~\ref{prop:Gauss}(a) is $1$ (the single fixed point $\emptyset$ under the full group action), not $a_d(1)=2$.
\end{remark}

A natural question that arises is whether Corollary~\ref{cor:zeta-primepower} is true not just for prime powers. This question leads us to the following definition:

\begin{definition}
For every $d \geq 1$ and for every $n \geq 1$, we set $N_d(n)=\tfrac1n\sum_{r\mid n}\mu(\frac{n}{r})a_d(r)\in\Q$.
\end{definition}

Therefore, we would like to know whether $N_d(n)$ is an integer. The following lemma relates the integrality of $N_d(n)$ to the integrality of a certain power series:

\begin{lemma}\label{lem:zetaformal}
 The formal power series $\zeta_d(x)=\prod_{n\ge1}(1-x^n)^{-N_d(n)}$ has integer coefficients (i.e. $\zeta_d (x) \in \Z[[x]]$) if and only if $N_d(n)\in\Z$ for every $n\ge1$.
\end{lemma}
\begin{proof}
By M\"obius inversion we have that $a_d(n)=\sum_{e\mid n}eN_d(e)$, and so by substituting and reindexing the double sum $\sum_{n}a_d(n)\frac{x^n}{n}=\sum_n\sum_{e\mid n}eN_d(e)\frac{x^n}{n}$ by writing $n=ej$, we get that it equals $\sum_{e\geq 1}\sum_{j\ge1}N_d(e)\frac{x^{ej}}{j}=-\sum_eN_d(e)\log(1-x^e)$, so $\zeta_d(x)=\exp\big(-\sum_eN_d(e)\log(1-x^e)\big)=\prod_e(1-x^e)^{-N_d(e)}$ as formal power series. If $N_d(n)\in\Z$ for every $n$, each factor $(1-x^n)^{-N_d(n)}$ lies in $\Z[[x]]$ (a polynomial if $N_d(n)<0$, and an integer binomial series if $N_d(n)\ge0$), so $\zeta_d\in\Z[[x]]$. Conversely, if $\zeta_d\in\Z[[x]]$, by matching coefficients of $x^n$ on both sides of $\zeta_d(x)=\prod_e(1-x^e)^{-N_d(e)}$, in increasing order of $n$, we can determine $N_d(n)$ inductively as an integer combination of the (integer) Taylor coefficients of $\zeta_d$ and of $N_d(1),\dots,N_d(n-1)$, and so $N_d(n)\in\Z$.
\end{proof}

\begin{remark}\label{rem:Arnold}
For a fixed integer matrix $A$, the identity $\log\det(I-xA)=-\sum_n\tr(A^n)\frac{x^n}{n}$ gives us the equality $\exp\big(\sum_n\tr(A^n)\frac{x^n}{n}\big)=\det(I-xA)^{-1}$, which is the mechanism behind the Arnold-Zarelua congruence (see~\cite{Arnold2006,Zarelua2006,Zarelua2008}) for $\tr(A^n)$, and also behind the classical fact that $N_d(n)\in\Z$ automatically whenever $a_d(n)=\tr(A^n)$ for every $n\ge1$, with a single matrix $A$, independent of $n$.% In our case at $d=1$ we have that  $M_1(n)=M=\begin{bsmallmatrix}1&1\\1&0\end{bsmallmatrix}$ is indeed independent of $n$ (as in Example~\ref{ex:lucas}), and $\tr(M^n)=L_n$ for every $n\ge1$ (including $n=1$: $\tr(M)=1=L_1$), but this is not quite $a_1(n)$, since $a_1(1)=2\ne L_1=1$ by the single-vertex convention of Definition~\ref{def:torus} (which is why Proposition~\ref{prop:trace} is stated only for $n\ge2$). Writing $a_1(n)=L_n+\delta_{n,1}$, we get $\sum_na_1(n)x^n/n=\sum_nL_nx^n/n+x=-\log\det(I-xM)+x=-\log(1-x-x^2)+x$, so
%\[
%\zeta_1(x)\ =\ \frac{e^x}{1-x-x^2},
%\]
%not the naively expected $\det(I-xM)^{-1}=1/(1-x-x^2)$ (the Fibonacci generating function $\sum_nF_{n+1}x^n$): the single-vertex convention $a_d(1)=2$ contributes an extra factor $e^x$, and consequently $\zeta_1$ is itself already transcendental as a function (though with the same finite radius of convergence $1/\varphi$ as $1/(1-x-x^2)$, since $e^x$ is entire) - the "fixed matrix $\Rightarrow$ rational $$" mechanism does not literally apply even at $d=1$, once one uses the paper's own indexing convention. Proposition~\ref{thm:zeta} shows the situation only worsens for $d\ge2$: there, $M_d(n)$ genuinely grows with $n$ (Definition~\ref{def:transfer}), so there is not even a candidate fixed matrix for the Arnold-Zarelua mechanism to apply to, and $\zeta_d$ has radius of convergence $0$ (an even more severe failure than the mere transcendence of $\zeta_1$). }
\end{remark}

As we saw in Remark~\ref{rem:k1exception}, we have a problem with integrality of $N_d(n)$ in the case where $n=p$ is prime since $a_d(1)=2$ and not $1$. Therefore, one might ask whether changing this first value solves the problem. For this goal, we define the sequence $\{\widehat a_d(n)\}_{n=1}^\infty$ by setting $\widehat a_d(1)=1$ and $\widehat a_d(n)=a_d(n)$ for every $n>1$. Similarly, we set $\widehat N_d(n)=\frac1n\sum_{e\mid n}\mu(\frac{n}{e})\widehat a_d(e)$ and $\widehat\zeta_d(x) =
\prod_{n\ge1}(1-x^n)^{-\widehat N_d(n)}$. The following proposition, in fact,  tells us that $\{\widehat a_d(n)\}_{n=1}^\infty$ satisfies an even stronger result than Corollary~\ref{cor:zeta-primepower}:

\begin{proposition}\label{thm:full-dold}
Let $d\ge1$ and let $n\ge2$. Then for every prime $p \mid n$  we have that 
\begin{enumerate}
    \item $\widehat a_d(n)\equiv\widehat a_d(\frac{n}{p})\pmod{p^{r}}$, where $r={\nu_p(n)}$. 
    \item $P_d(n)=\sum_{e\mid n}\mu(\frac{n}{e})\widehat a_d(e)
\in n\Z_{\ge0}$. 
\end{enumerate}
\end{proposition}

\begin{proof}
As in Theorem~\ref{prop:Gauss}, let $G_n=(\sfrac{\Z}{n\Z})^d$ act by translations on
$X_n=\Ind(T_d(n))$. Write $n=pm$ and consider $K=mG_n\cong(\sfrac{\Z}{p\Z})^d$. If $m\ge2$, then from Lemma~\ref{lem:pullback} we have an isomorphism $\Fix_K(X_n)\cong X_m$ of $G_n$-spaces, so $|\Fix_K(X_n)|=a_d(m)=\widehat a_d(m)$. If $m=1$, then $K=G_n$ and the only $G_n$-invariant independent set is the empty configuration, so again $|\Fix_K(X_n)|=1=\widehat a_d(1)$. \\

Since $G_n$ is abelian, then the complement $X_n \setminus \Fix_K(X_n)$ is a union of $G_n$-orbits. Let $P$ be the $p$-Sylow subgroup of $G_n$. Then $P\cong(\sfrac{\Z}{p^r\Z})^d$, and up to multiplication by a unit, we have that  $K=p^{r-1}P$. For $I\notin X_n^K$ we have that $K\nsubseteq\Stab_{G_n}(I)$, and by applying Lemma~\ref{lem:subgroup} to $P$ and $P\cap\Stab_{G_n}(I)$, we can conclude that $p^r\mid[P:P\cap\Stab_{G_n}(I)]$.  This index divides $[G_n:\Stab_{G_n}(I)]$, the size of the $G_n$-orbit of $I$. Therefore every orbit in the complement has cardinality divisible by $p^r$, and therefore we can conclude that $|X_n|-|X_n^K|\equiv0\pmod{p^r}$, which gives us item 1.\\

For item 2, fix some $n$. For $I\in X_n$, let $r(I)$ be the exponent of $\sfrac{G_n}{\Stab_{G_n}(I)}$ (i.e. the smallest $l \in \mathbb{N}$ such that $g^l \in \Stab_{G_n}(I)$ for every $g \in G_n$). Note that $r(I)$ is always a divisor of $n$ and for every $e\mid n$ we have that $eG_n\subseteq\Stab_{G_n}(I)$ if and only if $r(I)\mid e$. In addition, the number of configurations fixed by $eG_n$ is $\widehat a_d(e)$, since for $e\ge2$ these are precisely pullbacks from $T_d(e)$, and for $e=1$ only the empty configuration is fixed. Therefore, if we set $Q_d(r,n)$ to be the number of configurations $I\in X_n$ with $r(I)=r$, then from the previous argument we get that $\widehat a_d(e)=\sum_{r\mid e}Q_d(r;n)$.  Thus, by performing the Möbius inversion at $e=n$, we can conclude that
\[
Q_d(n,n)=\sum_{e\mid n}\mu\left(\frac{n}{e}\right)\widehat a_d(e)=P_d(n).
\]
Partition the configurations counted by $Q_d(n,n)$ into translation orbits. Such an orbit has cardinality $|G/\Stab_G(I)|$. The exponent of this quotient is $n$, and the exponent of a finite group divides its order. Therefore every such orbit has cardinality divisible by $n$. Therefore $P_d(n)=Q_d(n,n)\in n\Z_{\ge0}$ as desired. 
\end{proof}

\begin{remark}
    Proposition~\ref{thm:full-dold} tells us that $\{\widehat a_d(n)\}_{n=1}^\infty$ is a nonnegative Dold sequence. Dold sequences arise in many fields and their relation to periodic-point counts is surveyed in \cite{byszewski2021dold}.
\end{remark}

\begin{remark}\label{rem:gener_action}
    The proof of Theorem~\ref{prop:Gauss} and of Proposition~\ref{thm:full-dold} can be used to prove a more general observation: Let $G_k=(\quot{\Z}{p^k\Z})^d$ and let $X_k$ be a finite $G_k$-set such that $X_k^{p^{k-1}G_k}$ can be identified with $X_{k-1}$ (as a $G_{k-1}$-set). Then $|X_k| \equiv |X_{k-1}| \mod p^k$ for every $k$. In addition, we can show that if $X\subseteq\mathcal A^{\Z^d}$ is a finite-alphabet shift of finite type, and $b_X(n)$ denotes the number of points fixed by $n\Z^d$, then $\sum_{e\mid n}\mu(n/e)b_X(e)\in n\Z_{\ge0}$ for every $n\ge1$. 
\end{remark}

Proposition~\ref{thm:full-dold} settles  the integrality question of $N_d(n)$:

\begin{corollary}[Canonical integral zeta function]\label{cor:integral-zeta}
For every $d \geq 1$ and $n \geq 1$ we have that 
\begin{enumerate}
    \item $\widehat N_d(n)\in\Z_{\ge0}$. 
    \item $\widehat\zeta_d(x)
\in\Z[[x]]$. 
    \item $\zeta_d(x)=e^x\widehat\zeta_d(x)$.
    \item $N_d(n)=\widehat N_d(n)+\frac{\mu(n)}n$. In particular, $N_d(n)\in\Z$ if and only if $n$ is either $1$ or not squarefree.
\end{enumerate}
\end{corollary}

\begin{proof}
Items 1 and 2 follow from Proposition~\ref{thm:full-dold} (together with an analogue of the proof of Lemma~\ref{lem:zetaformal}). Since $a_d$ and $\widehat{a}_d$ agree for every $n>1$ and differ  by exactly one at $n=1$, the exponential generating series differ by $x$, giving us item 3. The same observation gives us the equality $ N_d(n)-\widehat N_d(n)=\frac{\mu(n)}n$. Finally, if $n>1$ is squarefree, then $\mu(n)=\pm1$, so this number is not an integer. If $n$ is not squarefree, then $\mu(n)=0$.
\end{proof}

%Lemma~\ref{lem:zetaformal} reduces the question of whether $\zeta_d\in\Z[[x]]$ to whether $N_d(n)\in\Z$ for every $n$, equivalently whether $n\mid\sum_{e\mid n}\mu(n/e)a_d(e)$ for every $n$. This is precisely the kind of statement proved, unconditionally, at prime-power $n=p^k$ with $k\ge2$ by Theorem~\ref{prop:Gauss} below (see Corollary~\ref{cor:zeta-primepower}) - but, as Remark~\ref{rem:k1exception} shows, it is genuinely false at $k=1$, whether it holds once $n$ has two or more distinct prime factors is a separate question our methods do not resolve, and none of our results depend on it. 

The following proposition tells us that even though $\zeta_d$ contains information about the arithmetic properties of $a_d(n)$, we cannot extract analytic information by regarding $\zeta_d$ as an analytic  function on an Archimedean domain:

\begin{proposition}\label{thm:zeta}
For $d\ge2$ the function $\zeta_d(x)$ has archimedean radius of convergence $0$. In particular, $\zeta_d$ is not a rational function of $x$.
\end{proposition}
\begin{proof}
By Lemma~\ref{lem:naive_bound} we have that $a_d(n)\ge2^{n^d/2}$ for every even $n$. Therefore we can conclude that $\limsup_na_d(n)^{\frac{1}{n}}=\infty$ (since $d\ge2$ and so $n^d/2\ge n$ for all $n\ge2$), and by the Cauchy-Hadamard theorem, the radius of convergence of $\sum_n\left(\frac{a_d(n)}{n}\right)x^n$ is $0$. Since $\zeta_d(0)=1\ne0$ we have that $\zeta_d$ is zero-free near $0$, and so on some neighborhood of $0$ the principal branch of $\log\zeta_d$ agrees with the power series $\sum_n\left(\frac{a_d(n)}{n}\right)x^n$. By the definition of $\zeta_d$ as its exponential, a positive radius of convergence for $\zeta_d$ around a zero-free point would force one for $\log\zeta_d$ there too, contradicting the above. Since every rational function with nonzero constant term is holomorphic, thus has a positive radius of convergence, on a neighborhood of $0$, and thus $\zeta_d$ cannot be rational.
\end{proof}

\begin{remark}
    Proposition~\ref{thm:zeta} is false in the case $d=1$. Specifically, from Example~\ref{ex:lucas} we have that $a_1(n)=L_n$ for every $n$, and so we can directly compute that $\zeta_1(x)=\frac{e^x}{1-x-x^2}$, which has a positive archimedean radius. 
\end{remark}

\begin{remark}
    Proposition~\ref{thm:zeta} tells us that we cannot apply the classical Borel-Dwork rationality criterion (see~\cite{Borel1894,Dwork1960}) to $\zeta_d$ (for $d \geq 2$), which certifies rationality of a power series with integer coefficients from an archimedean radius $R$ exceeding, at every prime $p$, a $p$-adic meromorphy radius $r_p^{-1}$. Even granting $N_d(n)\in\Z$ for every $n$ (so $\zeta_d\in\Z[[x]]$, and so $r_p\ge1$ at every $p$), Proposition~\ref{thm:zeta} gives us that $R=0$ unconditionally for $d\ge2$, so $R>r_p^{-1}$ can never hold at any $p$.
\end{remark}

Theorem~\ref{prop:Gauss} and Corollary~\ref{cor:p_mod_1} together give us that $a_d(p^k)-1$ is always divisible by $p$, so we now turn to the finer question of when the exact power of $p$ dividing $a_d(p^k)-1$ is independent of $k$.

\begin{definition}\label{def:tame}
We say that a prime $p$ is \textbf{tame with respect to $d$} if $\nu_p(a_d(p^k)-1)=\nu_p(a_d(p)-1)$ for every $k\ge1$. We denote the set of primes tame with respect to $d$ by $\Tame_d$.
\end{definition}

The following proposition gives us a useful criterion for tameness that we use throughout:

Thus tameness is not merely computable level by level: it is decided after at most $m$ levels, where $m$ is already known from $a_d(p)$.

\begin{proposition}\label{prop:tame_iff}
Let $d\geq1$. Then:
\begin{enumerate}
    \item If $\nu_p(a_d(p)-1)=1$ then $p \in \Tame_d$.
    \item If $\nu_p\big(a_d(p^j)-a_d(p^{j-1})\big) > \nu_p(a_d(p)-1) \geq 2$ for every $2\le j\le \nu_p(a_d(p)-1)$ then $p \in \Tame_d$. 
    \item $p$ is tame with respect to $d$ if and only if we have that  $\nu_p(a_d(p^k)-1)= \nu_p(a_d(p)-1)$ for every $1\le k\le \nu_p(a_d(p)-1)$.
\end{enumerate}
\end{proposition}
\begin{proof}
First, assume that $\nu_p(a_d(p)-1)=1$. As a telescopic series, we have that $\sum_{j=1}^k (a_d(p^j)-a_d(p^{j-1}))=a_d(p^k)-1$ for every $k\ge1$. Since $\nu_p(a_d(p)-1)=1$, and since by Theorem~\ref{prop:Gauss} we have that $\nu_p(a_d(p^j)-a_d(p^{j-1}))\ge j\ge2$ for every $j\ge2$. Since $\nu_p(a_d(p)-1)=1$ is strictly smaller than $\nu_p(a_d(p^j)-a_d(p^{j-1}))$ for every $j\ge2$, the non-archimedean triangle inequality (if $\nu_p(x)<\nu_p(y)$ then $\nu_p(x+y)=\nu_p(x)$) gives $\nu_p\big(\sum_{j=1}^k(a_d(p^j)-a_d(p^{j-1}))\big)=\nu_p(a_d(p)-1)=1$ for every $k$.\\ 

Second, suppose $\nu_p(a_d(p)-1) \ge2$. As in the previous case,  $a_d(p^k)-1= (a_d(p)-1) +  \sum_{j=2}^k(a_d(p^j)-a_d(p^{j-1}))$. Since $\nu_p(a_d(p^j)-a_d(p^{j-1}))>\nu_p(a_d(p)-1)$ for $2\le j\le \nu_p(a_d(p)-1)$, and by Theorem~\ref{prop:Gauss}, $\nu_p(a_d(p^j)-a_d(p^{j-1}))\ge j>\nu_p(a_d(p)-1) $ for $j>\nu_p(a_d(p)-1)$. So $\nu_p(a_d(p^j)-a_d(p^{j-1}))>\nu_p(a_d(p)-1)$ for every $j\ge2$, and the non-archimedean triangle inequality gives $\nu_p(a_d(p^k)-1)=\nu_p(a_d(p)-1)$ for every $k\ge1$.\\

Finally, for item 3, the forward implication is immediate. Conversely, suppose the displayed equality holds for $1\le j\le \nu_p(a_d(p)-1)$. For $k>m$, since $a_d(p^k)-a_d(p^m)= \sum_{j=\nu_p(a_d(p)-1)+1}^{k}(a_d(p^j)-a_d(p^{j-1}))$. Since $\nu_p(a_d(p^j)-a_d(p^{j-1}))\ge j\ge \nu_p(a_d(p)-1)+1$, we have that $a_d(p^k)\equiv a_d(p^{\nu_p(a_d(p)-1)})\pmod{p^{\nu_p(a_d(p)-1)+1}}$. Thus, since $\nu_p(a_d(p^{\nu_p(a_d(p)-1)})-1)=\nu_p(a_d(p)-1)$, it follows that $\nu_p(a_d(p^k)-1)=\nu_p(a_d(p)-1)$ for every $k>\nu_p(a_d(p)-1)$, as desired. 
\end{proof}

\begin{remark}\label{rem:WSS}
When $d=1$ from Example~\ref{ex:lucas} we have that $a_1(n)=L_n$, and so for a prime $p>5$, the classical definition of a (first-order, "Fibonacci-Wieferich" type) \textbf{Wall-Sun-Sun prime} is exactly a prime with $L_p\equiv1\pmod{p^2}$, i.e.\ with $m=\nu_p(a_1(p)-1)\ge2$ (the prime $5=\operatorname{disc}(\Q(\sqrt5)$ is excluded from the classical definition, as it ramifies and behaves differently), with the equivalence with the usual Fibonacci-quotient formulation $F_{p-(p/5)}\equiv0\pmod{p^2}$ is classical, see e.g.\ \cite{McIntoshRoettger} and the references therein). By Proposition~\ref{prop:tame_iff}, every prime $p> 5$ that is \textbf{not} Wall-Sun-Sun is tame for $d=1$. McIntosh and Roettger's original computation \cite{McIntoshRoettger} found no Wall-Sun-Sun prime below $2\times10^{14}$, subsequent distributed-computing searches have since pushed this bound substantially further, to primes smaller than $2^{64}\approx1.8\times10^{19}$ as of a completed PrimeGrid search in December 2022 (see~\cite{PrimeGrid}), again without finding one. So every prime that has been checked is tame for $d=1$, and whether Wall-Sun-Sun primes exist at all remains an open conjecture, and we do not resolve it here.
\end{remark}

% \begin{remark}\label{rem:tame_name}
% The terminology reflects that once the first exceptional level $m=\nu_p(a_d(p)-1)$ is $\ge2$, tameness of $p$ requires only finitely many further congruences (checked at levels $j=2,\dots,m$) to hold with \textbf{strict} inequality, by Theorem~\ref{prop:Gauss}, the inequality $\nu_p(a_d(p^j)-a_d(p^{j-1}))\ge j$ is automatic for $j>m$, so no further checking is needed beyond level $m$.
%\end{remark}

It is natural to ask if any tame primes actually exist. We show that $2$ and $3$ are not only tame for every $d$, but satisfy 
$\nu_p(a_d(p)-1)=1$. The main tool we use is a refinement of the orbit-counting argument of Proposition~\ref{prop:Gauss}, where we identify exactly which nonempty orbits of $\Ind(T_d(p))$ have size exactly $p$ (as opposed to $p^2$ or higher). We start with the following observation:

\begin{proposition}\label{prop:decompgeneral}
For every $d\ge1$ and odd prime $p$: the translation action of $G=(\quot{\Z}{p\Z})^d$ on $\Ind(T_d(p))$ has orbit sizes dividing $p^d$, a unique size-$1$ orbit $\{\emptyset\}$, and, writing $\orb_j$ for the number of orbits of size $p^j$ ($1\le j\le d$),
\[
a_d(p)-1\ =\ \sum_{j=1}^dp^j\orb_j.
\]
Consequently $\nu_p(a_d(p)-1)=1$ if and only if $p\nmid \orb_1$.
\end{proposition}
\begin{proof}
As in the proof of the first item of Proposition~\ref{prop:Gauss},  $G$ acts simply transitively on the vertex set, so $\emptyset$ is the unique $G$-fixed independent set, orbit-stabilizer gives orbit sizes dividing $p^d$, sum orbit sizes over $\Ind(T_d(p))\setminus\{\emptyset\}$, and reduce mod $p^2$ for the final statement.
\end{proof}

Therefore, in order to understand $\nu_p(a_d(p)-1)$ we need to understand the behavior of the sets with an orbit of size $p$. To do so we show that each such orbit is governed by the graph $\Cay(\sfrac{\Z}{p\Z},D)$ for some $D$.

\begin{definition}
Let $p$ be some prime and let $d \geq 1$. For a nonzero linear functional $\varphi:\F_p^d\to\F_p$, write $c_i=\varphi(e_i)$ and let $H=\ker\varphi\le G=(\quot\Z{p\Z})^d$, a subgroup of index $p$.
\begin{enumerate}
    \item We say that $\varphi$ (or $H$) is \textbf{bad} if $c_i=0$ for some $i$, and \textbf{good} otherwise. 
    \item We set $F(H)=|\Fix_H(\Ind(T_d(p)))|$.
    \item For good $\varphi$, let $D_\varphi=\{\pm\varphi(e_i):1\le i\le d\}\subseteq\F_p^\times$, and let $c(\varphi)=i(\Cay(\sfrac{\Z}{p\Z}, D_\varphi))$. 
\end{enumerate}
    
\end{definition}

The following lemma gives us the main combinatorial information to prove tameness of $2$ and $3$:

\begin{lemma}\label{lem:reduction}
We have:
\begin{enumerate}
\item If $H$ is bad, the only $H$-invariant independent set of $T_d(p)$ is $\emptyset$.
\item If $H$ is good, the map $I\mapsto\varphi(I)$ is a bijection from the $H$-invariant independent sets of $T_d(p)$ onto $\Ind(\Cay(\sfrac{\Z}{p\Z}, D_\varphi))$.
\item $c(\varphi)$ depends only on $H=\ker\varphi$ (equivalently, only on $\varphi$ up to scaling by $\F_p^\times$), not on the choice of $\varphi$ with that kernel.
\item $\displaystyle a_d(p)\ \equiv\ 1+\sum_{H\le G,\ [G:H]=p}\big(F(H)-1\big)\pmod{p^2}$.
\item $c(\lambda D)=c(D)$ for every $\lambda\in\F_p^\times$, consequently, for $D=D_\varphi$ (where $\varphi$ good), $r(D_\varphi)=\frac{(c(D_\varphi)-1)}{p}$ depends only on $H=\ker\varphi$.
\end{enumerate}
\end{lemma}
\begin{proof}
For item 1, since $I$ is $H$-invariant, $I+h=I$ for every $h\in H$, if $c_i=0$ then $e_i\in H=\ker\varphi$, so in particular $I+e_i=I$, i.e.\ $x\in I$ if and only if $x+e_i\in I$. If $I\ne\emptyset$, pick $x\in I$, then $x+e_i\in I$ too, but $x,x+e_i$ are adjacent in $T_d(p)$, contradicting independence of $I$. So $I=\emptyset$.\\

For item 2, suppose $H$ is good, so every $c_i\ne0$. If $I$ is $H$-invariant, it is a union of $H$-cosets, i.e.\ $I=\varphi^{-1}(S)$ for $S=\varphi(I)\subseteq\quot\Z{p\Z}$ (as $\sfrac{G}{H}\cong\quot\Z{p\Z}$ via $\varphi$). Suppose $x,y=x+e_i\in I$ are adjacent, then $s=\varphi(x)$ and $s+c_i=\varphi(y)$ both lie in $S$, so $S$ is not independent in $\Cay\quot\Z{p\Z},\{\pm c_i\})$ (as $c_i\ne0$, this is a genuine edge). Conversely if $I=\varphi^{-1}(S)$ for $S$ independent in this circulant, and $x,y\in I$ were adjacent in $T_d(p)$ via $y=x+e_i$, then $\varphi(x),\varphi(y)=\varphi(x)+c_i\in S$ would be adjacent in the circulant (as $c_i\ne0$), contradicting independence of $S$. Since every $H$-invariant set is of the form $\varphi^{-1}(S)$ for a unique $S$ (as $\varphi$ is surjective with kernel $H$), this establishes the bijection.\\

For item 3, if $\varphi'=\lambda\varphi$ for $\lambda\in\F_p^\times$, then $c_i'=\lambda c_i$, and $\{\pm c_i'\}=\lambda\cdot\{\pm c_i\}$, the map $x\mapsto\lambda^{-1}x$ is a graph isomorphism $\Cay\quot\Z{p\Z},\{\pm c_i'\})\to\Cay\quot\Z{p\Z},\{\pm c_i\})$ (as $x\sim y$ via $x-y=\pm c_i'$ corresponds to $\lambda^{-1}x-\lambda^{-1}y=\pm\lambda^{-1}c_i'=\pm c_i$), so the two circulants are isomorphic and $c(\varphi')=c(\varphi)$. Since any two functionals with the same kernel differ exactly by a scalar $\lambda\in\F_p^\times$, we can conclude $R$ depends only on $H$.\\

For item 4, as in the proof of the item 2 in Proposition~\ref{prop:Gauss}, every orbit of the translation action of $G$ on $\Ind(T_d(p))$ has size a power of $p$, a nonempty orbit has size exactly $p$ iff its stabilizer has index exactly $p$, i.e.\ is some index-$p$ subgroup $H$. Since $\sfrac{G}{H}\cong\quot\Z{p\Z}$ is cyclic of prime order, $H$ is a \textbf{maximal} subgroup, so for $I\ne\emptyset$ we have that $I$ is $H$-invariant if and only if $\Stab_G(I)\supseteq H$, which is equivalent to having $\Stab_G(I)=H$ (as $\Stab_G(I)\ne G$, else $I$ is either empty or everything, and the whole vertex set is not independent for $p\ge2$) or $\Stab_G(I)=G$, but $I\ne\emptyset$ so $\Stab_G(I) \ne G$, so indeed $\Stab_G(I)=H$ exactly. So the nonempty $H$-invariant independent sets are exactly those with orbit size $p$ and stabilizer $H$, by the previous items there are $F(H)-1$ of these (subtracting the empty set), and this number is automatically a multiple of $p$ (partition into full $G$-orbits of size $p$, since translating an $H$-invariant set by any $g\in G$ gives another $H$-invariant set as $H\trianglelefteq G$). Different subgroups $H\ne H'$ give disjoint sets of nonempty invariant independent sets. This is true since if $I$ were both $H$- and $H'$-invariant, $I$ would be invariant under $\langle H,H'\rangle$, which equals $G$ since $H,H'$ are distinct maximal subgroups, forcing $I=\emptyset$. Therefore we can conclude that 
\[
\sum_{\text{orbits of size }p}|\textup{Orbit}|\ =\ \sum_H\big(F(H)-1\big),
\]
and combined with $a_d(p)-1=\sum_{\text{nontrivial orbits}}|\textup{Orbit}|$ and the fact that all orbits of size $\ge p^2$ contribute $0\pmod{p^2}$.\\

For item 5, the claim $c(\lambda D)=c(D)$, and thus that $c(D_\varphi)$ (so also $r(D_\varphi)$) depends only on $H=\ker\varphi$, is exactly (iii). That $r(D_\varphi)\in\Z$ follows from the next paragraph. Since $\quot{\Z}{p\Z}$ (by rotation) acts on $\Ind(\Cay(\sfrac{\Z}{p\Z},D))$ for $D\ne\emptyset$ with orbit sizes $1$ or $p$, and only $\emptyset$ gives a fixed point (the full vertex set is never independent, $D\ne\emptyset$), we get $c(D)\equiv1\pmod p$, so $r(D)=\sfrac{(c(D)-1)}{p}\in\Z$, moreover $r(D)$ counts the $G$-orbits of independent sets $I$ of $T_d(p)$ with $\Stab_G(I)=\ker\varphi$, for $D=D_\varphi$.
\end{proof}

\begin{corollary}\label{thm:p2}
For every $d\ge1$ we have that:
\begin{enumerate}
    \item $a_d(2)\equiv3\pmod4$.
    \item $a_d(3)\equiv1+3\cdot2^{d-1}\pmod9$.
\end{enumerate}
In particular, $2$ and $3$ are tame for every $d$ with $\nu_2(a_d(2)-1)=\nu_3(a_d(3)-1)=1$.
\end{corollary}
\begin{proof}
For $p=2$ we have that $\F_2^\times=\{1\}$, and so every nonzero functional $\varphi$ has $c_i\in\{0,1\}$ with no scaling ambiguity, and there are $2^d-1$ index-2 subgroups (one per nonzero $(c_1,\dots,c_d)\in\F_2^d$). Therefore, exactly one of these is good, which is  $c_1=\dots=c_d=1$. For this $\varphi$ we have that $\{\pm c_i\}=\{1\}$, as $-1\equiv1\pmod2$, and so $c(\varphi)=i\big(\Cay\quot\Z{2\Z},\{1\})\big)=i(K_2)=3$ (the two vertices are adjacent, independent sets are $\emptyset,\{0\},\{1\}$). By part 4 of Lemma~\ref{lem:reduction} we have that $a_d(2)\equiv1+(3-1)=3\pmod4$ for every $d$.\\

For $p=3$ we have that $\F_3^\times=\{1,2\}$, and so a good $\varphi$ has $c_i\in\{1,2\}$ for every $i$, giving $2^d$ choices, in $2^{d-1}$ scaling classes (as $\F_3^\times =\{1,2\}$ with $-1=2$, so scaling by $\{1,2\}$ has orbits of size 2 on good vectors, therefore $2\ne1$ is a nontrivial scalar and no good vector is fixed by it). For every good $\varphi$ we have that $\{\pm c_i\}=\{1,2\}$ regardless of the specific values (since both $1,2\in\{1,2\}$ are already negatives of each other mod 3: $-1=2$), so the connection set is always all of $\F_3^\times=\{1,2\}$, giving that $\Cay\quot\Z{3\Z},\{1,2\})$ is the complete graph on 3 vertices and has $4$ independent sets. So $c(\varphi)=4$ for all $2^{d-1}$ good hyperplanes, and by part 4 of Lemma~\ref{lem:reduction}, $a_d(3)\equiv1+2^{d-1}(4-1)=1+3\cdot2^{d-1}\pmod9$. Since $3\nmid2^{d-1}$, we have that $\nu_3(3\cdot2^{d-1})=1$ exactly, and so $a_d(3)-1\equiv3\cdot2^{d-1}\not\equiv0\pmod9$.
\end{proof}

From Proposition~\ref{thm:p2} one would expect that $\nu_p(a_d(p)-1)=1$ for every $p$. Yet, this is not true, even for $p=5$, as we see in the following proposition:

\begin{proposition}\label{thm:p5}
$a_d(5)\equiv1+5\big(2^{d-1}+4^{d-1}\big)\pmod{25}$ for every $d\ge1$. Consequently, $\nu_5(a_d(5)-1)=1$ if and only if $d\not\equiv3\pmod4$.
\end{proposition}
\begin{proof}
For $p=5$ we have that $\F_5^\times=\{1,2,3,4\}$ splits into two $\pm1$-classes, namely $A=\{1,4\}$ and $B=\{2,3\}$ (as $4=-1$ and $3=-2$). For a good $\varphi$ with $c_i\in\F_5^\times$, the connection set $\{\pm c_i:1\le i\le d\}$ equals $A$ if every $c_i\in A$, equals $B$ if every $c_i\in B$, and equals $A\cup B=\F_5^\times$ if both classes occur among $c_1,\dots,c_d$ (as, e.g., $c_i=1\Rightarrow\{\pm c_i\}=A$ and $c_j=2\Rightarrow\{\pm c_j\}=B$, and these unions accumulate). The three possible circulants are: $\Cay\quot\Z{5\Z},A)=\Cay\quot\Z{5\Z},\{1,4\})\cong C_5$ (the $5$-cycle, since $1$ generates $\quot\Z{5\Z}$), with $i(C_5)=L_5=11$ (Example~\ref{ex:lucas}), $\Cay\quot\Z{5\Z},B)=\Cay\quot\Z{5\Z},\{2,3\})$, isomorphic to $C_5$ via $x\mapsto2^{-1}x$ (as $2$ also generates $\quot\Z{5\Z}$), so again $i=11$, and $\Cay\quot\Z{5\Z},\F_5^\times)$ is a complete graph on 5 vertices, which has $6$ independent sets.  \\

We now count good hyperplanes of each kind. There are $4^d$ good vectors $(c_1,\dots,c_d)$, in $4^{d-1}$ scaling classes (as $|\F_5^\times|=4$ acts freely on good vectors: $\lambda c=c$ forces $\lambda=1$ since $c_i\ne0$). Call a vector's class-pattern the tuple in $\{A,B\}^d$ recording which class each $c_i$ lies in, there are $2^d$ possible patterns, each realized by exactly $2^d$ vectors (2 choices within each class per coordinate). Scaling by $-1\in\{\pm1\}$ multiplies each $c_i$ by $-1$, which preserves each coordinate's class ($-1\cdot A=A$ and  $-1\cdot B=B$, since $A,B$ are themselves $\pm$-stable), hence preserves the class-pattern, scaling by $2$ (representing the other coset of $\{\pm1\}$ in $\F_5^\times$) sends $A\to B$ and $B\to A$ ($2\cdot1=2\in B$, $2\cdot4=3\in B$, $2\cdot2=4\in A$, $2\cdot3=1\in A$), so it sends every class-pattern to its complementary pattern $\bar\tau$ (every coordinate's class flipped). So the $4^{d-1}$ good hyperplanes correspond to unordered pairs $\{\tau,\bar\tau\}$ of class-patterns (there are $2^{d-1}$ such pairs, as no pattern is self-complementary for $d\ge1$), each pair accounting for $\frac{4^{d-1}}{2^{d-1}}=2^{d-1}$ hyperplanes (matching that each of the $2^d$ vectors realizing pattern $\tau$, or equally $\bar\tau$, splits into $2^{d-1}$ pairs under the $\{\pm1\}$-action, each combining with a $\{2,3\}$-translate to complete one hyperplane orbit of size 4).\\

Exactly one pair $\{\tau,\bar\tau\}$ is uniform (all-$A$ paired with all-$B$, its complement), contributing $2^{d-1}$ hyperplanes with $c=11$ each, the remaining $2^{d-1}-1$ pairs are mixed (both $\tau,\bar\tau$ hit both classes, since a pattern that is not uniform, upon complementing every coordinate, remains non-uniform), each contributing $2^{d-1}$ hyperplanes with $c=6$. By part 4 of Lemma~\ref{lem:reduction} we have that
\[
a_d(5)-1\ \equiv\ 2^{d-1}(11-1)+\big(2^{d-1}-1\big)2^{d-1}(6-1)\  =\ 5\cdot2^{d-1}\big(2^{d-1}+1\big)\pmod{25}.
\]
Since $5\cdot2^{d-1}(2^{d-1}+1)=5\cdot2^{d-1}+5\cdot4^{d-1}$, this is $5(2^{d-1}+4^{d-1})\pmod{25}$, giving the claimed congruence. As $4\equiv-1\pmod5$, we have that $2^{d-1}+4^{d-1}\equiv2^{d-1}+(-1)^{d-1}\pmod5$, since $\ord_5(2)=4$, this quantity cycles with period 4 in $d$, taking the values $2,1,0,2\pmod5$ at $d\equiv1,2,3,0\pmod4$ respectively (direct check), vanishing exactly when $d\equiv3\pmod4$.
\end{proof}

We end this section with a closer analysis for $d=2$, both as an illustration of the mechanism above and because it isolates precisely the additional (open) arithmetic input that would be needed to decide tameness of a prime $p$ once $\nu_p(a_2(p)-1)\ge2$. In order to do so, by Proposition~\ref{prop:decompgeneral}, we wish to understand $\orb_1$. 

% Throughout this subsection, $p$ is an odd prime, $G=(\quot{\Z}{p\Z})^d$ acts on $X=\Ind(T_d(p))$ by translation, and we set
% \[
% Q_d(p)\ =\ \frac{a_d(p)-1}p\ \in\ \Z,
% \]
% so that $\nu_p(a_d(p)-1)=1$ if and only if $Q_d(p)\not\equiv0\pmod p$. As in the proof of Proposition~\ref{prop:Gauss}(a), every nontrivial $G$-orbit has size a power of $p$ dividing $p^d$, so
% \[
% Q_d(p)\ \equiv\ \#\{\text{orbits of size exactly }p\}\pmod p.
% \]
% Orbits of size exactly $p$ correspond to independent sets $I$ with $\Stab_G(I)=H=\ker\varphi$ for some hyperplane, i.e.\ nonzero linear functional $\varphi:\F_p^d\to\F_p$. Call $\varphi$ (equivalently, $H$) \textbf{good} if $\varphi(e_i)\ne0$ for every $i$, and \textbf{bad} otherwise. .

\begin{proposition}\label{lem:mobius}
If $d=2$ we have $$\orb_1\ =\ \frac1p\sum_{[G \colon H]=p}\big(F(H)-1\big)=\ 2L_p-(p-1)+\sum_{r\in\F_p^\times\setminus\{\pm1\}} i\big(\Cay(\sfrac{\Z}{p\Z},\{\pm1,\pm r\})\big).$$ 

\end{proposition}
\begin{proof}
Since $[G:H]=p$ is prime, no subgroup lies strictly between $H$ and $G$, so $\Stab_G(S)\supseteq H$ forces $\Stab_G(S)=H$ or $\Stab_G(S)=G$ (the latter meaning $S=\emptyset$, the only $G$-fixed point). Hence $N(H)=|\{S:\Stab_G(S)=H\}|=F(H)-1$. As $G$ is abelian, an orbit with stabilizer $H$ has size $p$ and contributes $p$ independent sets to $N(H)$, so the number of such orbits is $\sfrac{(F(H)-1)}{p}$, summing over the $p+1$ index-$p$ subgroups (i.e.\ hyperplanes of $\F_p^2$) gives $\orb_1$.\\

Now, every good hyperplane has a unique representative $\phi=(1,r)$ for some $ r\in\F_p^\times$. Therefore $p \cdot \,\operatorname{orb}_1
=
\sum_{r\in\F_p^\times}\bigl(i\big(\Cay(\sfrac{\Z}{p\Z},\{\pm1,\pm r\})\big)-1\bigr)$. Since $i\big(\Cay(\sfrac{\Z}{p\Z},\{\pm1,\pm 1\})\big)=L_p$, we have that $p\,\operatorname{orb}_1 = 2L_p-(p-1) + \sum_{r\ne\pm1}i\big(\Cay(\sfrac{\Z}{p\Z},\{\pm1,\pm r\})\big)$, as desired. 
\end{proof}

\begin{remark}
    We can observe that the set $\{\pm1,\pm r\}$ is unchanged under $r\mapsto-r$. And $x\mapsto r^{-1}x$ is a graph automorphism of $\quot{\Z}{p\Z}$ carrying $\{\pm1,\pm r\}$ to $\{\pm r^{-1},\pm1\}$. In particular, the group $\{\pm1\}\times\{1,\mathrm{inv}\}\cong(\sfrac{\Z}{2\Z})^2$ acts on $\F_p^\times\setminus\{\pm1\}$ by $r\mapsto\pm r^{\pm1}$, with orbits of size dividing $4$ (size $<4$ occurring only if $r^2=-1$, i.e.\ only when $p\equiv1\pmod4$). Consequently the sum $\sum_{r\ne\pm1} i\big(\Cay(\sfrac{\Z}{p\Z},\{\pm1,\pm r\})\big)$ is a sum of integers with at most $\left\lceil\frac{p-3}{4}\right\rceil$ genuinely distinct values of $i\big(\Cay(\sfrac{\Z}{p\Z},\{\pm1,\pm r\})\big)$, each occurring  $4$ times.
\end{remark}

% \begin{openquestion}
% Is there a closed form, or a nontrivial congruence, for $J(1,r)\bmod p$ - for a single $r$, or for the reduced sum of Corollary~\ref{cor:reduced}?
% \end{openquestion}

\begin{remark}
Writing $L_p=1+p\ell$ and $i\big(\Cay(\sfrac{\Z}{p\Z},\{\pm1,\pm r\})\big)=1+pj_r$ (both integers by Corollary~\ref{cor:p_mod_1} applied to $\Cay(\sfrac{\Z}{p\Z},\{\pm1\})=T_1(p)$ and $\Cay(\sfrac{\Z}{p\Z},\{\pm1,\pm r\})$, respectively, the same orbit-counting argument as item 1 of Proposition~\ref{prop:Gauss}, applied to the corresponding circulant graphs, which needs no new input),
\[
\nu_p(a_2(p)-1)=1\quad\Longleftrightarrow\quad 2\ell+\sum_{r\ne\pm1}j_r\ \not\equiv\ 0\pmod p.
\]
Whether $\ell\ne0$ (equivalently, whether $p$ fails to be Wall-Sun-Sun, in the sense of Remark~\ref{rem:WSS}) is one summand in this criterion, but $\ell\ne0$ is neither necessary nor sufficient for $\nu_p(a_2(p)-1)=1$: the second summand $\sum_{r \neq \pm 1} j_r$ is a structurally independent unknown that could, in principle, cancel it exactly.
\end{remark}

% \begin{remark}[The case $d=2$]\label{rem:d2}
% For $d=2$, the full subgroup lattice of $G=(\quot\Z{p\Z})^2$ (trivial, the $p+1$ subgroups of order $p$, and $G$ itself) is finite and explicit, and the orbit-counting of Lemma~\ref{lem:reduction} can be pushed further: each of the $p+1$ order-$p$ subgroups of $G$ is the kernel of a functional $\varphi_{[a:b]}(x,y)=bx-ay$ for a point $[a:b]\in\mathbb P^1(\F_p)$, and (translating the argument of Lemma~\ref{lem:reduction} to the full, not just index-$p$, subgroup lattice) the exact value of $a_2(p)$ decomposes as a sum, over all $p+1$ points of $\mathbb P^1(\F_p)$, of independent-set counts of explicit circulant graphs on $\quot\Z{p\Z}$ together with the trivial orbit at $\emptyset$. This is consistent with, and refines, theorem~\ref{thm:p3} and Theorem~\ref{thm:p5} at $d=2$ (which recover the congruences satisfied by this exact sum modulo $p^2$), and underlies the numerically verified values $a_2(3)=34$ and $a_2(5)=25531$. We do not pursue the resulting closed form for general $p$ further here, as it plays no role in the results of Sections~\ref{sec:padic}-\ref{sec:Ed}.
% \end{remark}

\section{$p$-adic Numbers and Algebraicity}\label{sec:padic}

 This section studies what a hypothetical algebraicity of $\kappa_d$ would force, by comparing the archimedean sequence $\big \{a_d(n)^{\frac{1}{n^d}}\big\}_{n=1}^\infty$ - which converges to $\kappa_d$ (Theorem~\ref{thm:existence}) - against a fixed algebraic value, using the classical machinery from $p$-adic analysis and algebraic number theory. We start by looking at the $p$-adic companion of $\kappa$:

\begin{proposition}\label{prop:kappa_p_exists}
For every $d\ge1$ and for every prime $p$, the sequence $\{a_d(p^k)\}_{k=1}^\infty$ converges $p$-adically to $\kappa_d^{(p)}\in1+p\Zp$, with $a_d(p^k)\equiv\kappa_d^{(p)}\pmod{p^k}$ for every $k$. If, in addition, $p\in\Tame_d$ then $\nu_p(\kappa_d^{(p)}-1)=\nu_p(a_d(p)-1)$.
\end{proposition}
\begin{proof}
Fix $m\ge1$, by induction on $k\ge m$, using Theorem~\ref{prop:Gauss}, $a_d(p^k)\equiv a_d(p^m)\pmod{p^m}$ for every $k\ge m$, so $\{a_d(p^k)\}_{k=1}^\infty$ is Cauchy in $\Zp$, converging to $\kappa_d^{(p)}$ with $\kappa_d^{(p)}\equiv a_d(p^m)\pmod{p^m}$ for every $m$, taking $m=1$ and Corollary~\ref{cor:p_mod_1}, $\kappa_d^{(p)}\in1+p\Zp$. If $p\in\Tame_d$ then from Proposition~\ref{prop:tame_iff} we have that $r_p=\nu_p(a_d(p)-1)=\nu_p(a_d(p^k)-1)$ for every $k$, and by taking $m=r_p+1$, we can conclude that $\kappa_d^{(p)}-1\equiv a_d(p^{r_p+1})-1\pmod{p^{r_p+1}}$, and the right side has $\nu_p=r_p<r_p+1$, forcing $\nu_p(\kappa_d^{(p)}-1)=r_p$ too.
\end{proof}

\begin{example}\label{ex:A1alg}
For $d=1$ and $p\ne5$, by Example~\ref{ex:lucas} we have that $a_1(p^k)=\alpha^{p^k}+\beta^{p^k}$, where $\alpha,\beta\in\mathcal O_L^\times$ are the solutions to $x^2-x-1$ in a fixed embedding of $K=\Q(\sqrt5)$ into $L=K_\mathfrak p$ ( for some unramified $\mathfrak p\mid p$, as $p\ne5$, with $f=[L:\Qp]\in\{1,2\}$ that is either split or inert according to $p\equiv\pm1$ or $\pm2\pmod5$). If $f=1$ then  $\omega(\alpha)^{p^k}=\omega(\alpha)$ for every $k$, and so $\langle\alpha\rangle^{p^k}\to1$ (by Proposition~\ref{prop:teich} below), which gives us that $a_1(p^k)\to\omega(\alpha)+\omega(\beta)$. If $f=2$ then swapping $\alpha$ and  $\beta$ is the local Frobenius, giving us $\omega(\beta)=\omega(\alpha)^p$, and by a direct computation based upon the splitting on the parity of $k$ and upon the fact that $p^2\equiv1\pmod{p^2-1}\ge\ord(\omega(\alpha))$, we again get that $\omega(\alpha)^{p^k}+\omega(\beta)^{p^k}=\omega(\alpha)+\omega(\beta)$ for every $k$. Either way, $\kappa_1^{(p)}=\omega(\alpha)+\omega(\beta)$ is algebraic as a sum of two roots of unity.
\end{example}

\begin{remark}
    Example~\ref{ex:A1alg} shows us that $\kappa_d^{(p)}$ need not be  $\kappa_d$ "read $p$-adically under any fixed embedding", even in the $d=1$ case (which is known to be algebraic). Specifically, at $p=11$ (which is a split prime as $4^2\equiv5\pmod{11}$), from Proposition~\ref{prop:kappa_p_exists} we have that $\kappa_1^{(11)} \equiv 1 \mod11$, yet the solutions of the equation $x^2-x-1$ are either $4$ or $8 \mod 11$, and so $\kappa_1^{(11)}$ cannot be the image of $\kappa_1=\frac{1+\sqrt{5}}{2}$ under any embedding $\Q(\sqrt{5}) \hookrightarrow \Qp$. For $d>1$ the situation is more extreme, since $M_d(p^k)$ genuinely grows with $k$, so there is no fixed algebraic structure underlying the sequence, and no a priori reason for $\kappa_d^{(p)}$ to be algebraic at all.
\end{remark}

One might hope to build a $p$-adic analogue of $\kappa_d$ directly, as the $p$-adic limit of $a_d(p^k)^{\frac{1}{p^{kd}}}$, matching $\kappa_d=\lim_na_d(n)^{\frac{1}{n^d}}$. This fails termwise: Proposition~\ref{prop:rootexist} below shows that whenever $\nu_p(a_d(p)-1)=1$ (e.g.\ $p=2,3$, any $d$, by Corollary~\ref{thm:p2}), $a_d(p^k)$ has no $p^{kd}$-th root in $\Zp$ at all, so the individual terms are not even defined.

\begin{proposition}\label{prop:rootexist}
Let $p$ be odd. Then:
\begin{enumerate}
    \item $a_d(p^k)$ has a $p^{kd}$-th root in $\Zp$ if and only if $\nu_p(a_d(p^k)-1) \geq kd+1$. When it exists, it is unique and lies in $1+p\Zp$,
    \item If $\nu_p(a_d(p)-1)=1$ then $a_d(p^k)$ has no $p^{kd}$-th root in $\Zp$, for any $k\ge1$.
\end{enumerate}
In addition, if $u \in \Z_2^\times$ and $\nu_2(u-1)=1$ then $u \equiv 3 \mod 4$ and has no $2^r$-th root in $\Z_2$ for every $r \geq 1$. 
\end{proposition}
\begin{proof}
Given some $u \in 1+p\Zp$, write $x=u-1$ and $s=\nu_p(x)\ge1$. In the power series $\log_p(1+x)=\sum_{j\ge1}(-1)^{j-1}\frac{x^j}{j}$, the $j$-th term has valuation $js-\nu_p(j)$, and for $j\ge2$ we have that $\nu_p(j)\le\log_pj<j-1\le(j-1)s$ (as $p$ is odd), and so $js-\nu_p(j)>s$, only $j=1$ attains the minimum, giving $\nu_p(\log_p u)=s$. Now, for the root criterion (assuming $p$ is odd), we have that any root $v$ of $v^{p^r}=u$ decomposes as $v=\xi w$ where  $\xi$ is a $(p-1)$-th root of unity and $w\in1+p\Zp$. Since $\xi^{p^e}w^{p^r}=u\in1+p\Zp$, we must have that $\xi^{p^r}=1$, and as $\gcd(p^r,p-1)=1$, we must have that $\xi=1$. On $1+p\Zp$, the function $\log_p$ is a group isomorphism onto $(p\Zp,+)$ (as $p$ is odd), so $w\mapsto w^{p^r}$ corresponds to injective multiplication by $p^r$, giving uniqueness, and $v=\exp\left(\frac{\log_p(u)}{p^r}\right)$ is well-defined and solves the equation $v^{p^r}=u$ exactly when $\frac{\log_p(u)}{p^r}\in p\Zp$, that is,  $\nu_p(u-1)=\nu_p(\log_p u)\ge r+1$. Thus, for $u\in1+p\Zp$ we have that $\nu_p(\log_p u)=\nu_p(u-1)$, and so for every $r\ge1$ we have that $u$ has a $p^r$-th root in $\Zp$ iff $p^{r+1}\mid u-1$, in which case the root is unique and lies in $1+p\Zp$. The $p=2$ case is similar. Therefore we can apply this argument to  $u=a_d(p^k)$ and $r=kd$. (since then $\nu_p(a_d(p^k)-1)=1<kd+1$ by Proposition~\ref{prop:tame_iff}). 
\end{proof}

The following proposition shows us how we can use Proposition~\ref{prop:trace} to construct a matrix analogue of Proposition~\ref{prop:kappa_p_exists} based upon the Arnold-Zarelua congruence for traces of a fixed matrix, as in Remark~\ref{rem:Arnold}:

\begin{proposition}\label{prop:beta}
For every prime $p$ and every $d \geq 1$ we have that the sequence $\big\{\tr(M_d(p^k)^{p^j})\big\}_{j=1}^\infty$ converges $p$-adically to $\beta_d^{(p)}(k)\in\Zp$, with $\tr(M_d(p^k)^{p^j})\equiv\beta_d^{(p)}(k)\pmod{p^j}$ for every $j$. In addition, we have that $\kappa_d^{(p)}\equiv\beta_d^{(p)}(k)\pmod{p^k}$ for every $k$.
\end{proposition}
\begin{proof}
By Proposition~\ref{prop:trace} we have that $a_d(p^k)=\tr(M_d(p^k)^{p^k})$, and so by telescoping the Arnold-Zarelua congruences $\tr(A^{p^{i+1}})\equiv\tr(A^{p^i})\pmod{p^{i+1}}$ for $i\ge k$ gives us that $\tr(A^{p^j})\equiv\tr(A^{p^k})\pmod{p^k}$ for every $j\ge k$, letting $j\to\infty$, we get that $\beta_d^{(p)}(k)\equiv a_d(p^k)\pmod{p^k}$, and Proposition~\ref{prop:kappa_p_exists} gives $\kappa_d^{(p)}\equiv a_d(p^k)\pmod{p^k}$ as well.
\end{proof}

\begin{remark}\label{rem:beta_teich}
The limit $\beta_d^{(p)}(k)$ of Proposition~\ref{prop:beta} has an explicit description that shows it is, in a precise sense, the $p$-adic shadow of an honest algebraic number, even though $M_d(p^k)$ is not a fixed matrix as $k$ varies. Write $\overline{M_d(p^k)}\in\Mat_r(\F_p)$ for the reduction of $M_d(p^k)$ mod $p$, and let $\lambda_1,\dots,\lambda_r\in\overline{\F_p}$ be its non-zero eigenvalues with multiplicity, each $\lambda_i$ lying in a finite field $\F_{p^{f_i}}$ for some $f_i$. By performing a Jordan decomposition on $\overline{M_d(p^k)}$ we can write $\overline{M_d(p^k)}=D+N$ where $D$ semisimple and $N$ nilpotent that satisfy $DN=ND$, the Frobenius identity $(x+y)^{p^j}=x^{p^j}+y^{p^j}$ for commuting $x,y$ in characteristic $p$ gives $\overline{M_d(p^k)}^{p^j}=D^{p^j}+N^{p^j}=D^{p^j}$ once $p^j$ exceeds the nilpotency index of $N$ (which happens for all large $j$, as $r$ is fixed), so $\tr(\overline{M_d(p^k)}^{p^j})=\sum_i\lambda_i^{p^j}$ for all large $j$. Zarelua's analysis \cite{Zarelua2008} refines this mod-$p$ picture to the full $p$-adic limit $\beta_d^{(p)}(k)\ =\ \sum_{i=1}^r\omega(\lambda_i)$, which is the sum of the Teichm\"uller lifts of the non-zero eigenvalues of $\overline{M_d(p^k)}$ (each $\omega(\lambda_i)$ a root of unity of order dividing $p^{f_i}-1$, lying in the unramified extension $\Qp(\xi_{p^{f_i}-1})$). In particular, by choosing some $\iota_p \colon \overline\Q\hookrightarrow\overline{\Q}_p$ compatible with the residue fields involved, $\beta_d^{(p)}(k)$ is the $p$-adic image of the honest algebraic number $\sum_i\xi_i\in\Q(\xi_{p^{f_1}-1},\dots,\xi_{p^{f_r}-1})\subset\overline\Q_p$ obtained by summing the corresponding complex roots of unity. Note that Example~\ref{ex:A1alg} is exactly the case where $d=1$ and $r=2$. \\
\end{remark}

Proposition~\ref{prop:beta} needs no combinatorial input at all - it applies to any fixed integer matrix, not only to $M_d(p^k)$. In particular it attaches a $p$-adic invariant to every finite-size transfer-matrix approximation of $\kappa_d$ used in the literature to obtain rigorous numerical bounds, including those built from strips rather than tori. For $d=2$, Pavlov studies the strip entropies $h_n=\log\rho(\widehat M_2(n))$, where $\widehat M_2(n)$ is our transfer matrix construction applied to the alphabet $\Ind(P_n)$ of the length-$n$ \textbf{path} (rather than our cycle $T_1(n)$), the general theory of $\Z^2$ shifts of finite type gives $\frac{h_n}{n}\to\log\kappa_2$, and Pavlov proves the sharper statement $h_{n+1}-h_n\to\log\kappa_2$, with at least exponential convergence rate \cite{Pavlov2012}. Writing $\eta_d(n)=\rho(\widehat M_d(n))$ for the analogous path-alphabet spectral radius in general dimension $d$ (so $\eta_2=\rho(\widehat M_2)$ above), the weaker, general-$d$ background fact already follows from our own machinery:

\begin{proposition}\label{prop:mu}
For every $d\ge2$ we have that $\eta_d(n)^{\frac{1}{n^{d-1}}}\to\kappa_d$ as $n\to\infty$.
\end{proposition}
\begin{proof}
Since $B_{d-1}(n)$ embeds into $(\quot{\Z}{(n+1)\Z})^{d-1}$ without creating extra adjacencies (Theorem~\ref{thm:existence}), every $B_{d-1}(n)$-independent set is also $T_{d-1}(n+1)$-independent, so $\Ind(B_{d-1}(n))\subseteq\Ind(T_{d-1}(n+1))$, and since $I\cap J=\emptyset$ means the same thing in either ambient graph, $\widehat M_d(n)$ is exactly the principal submatrix of $M_d(n+1)$ on this index subset. Symmetrically, $\Ind(T_{d-1}(n))\subseteq\Ind(B_{d-1}(n))$ (the torus has more edges, hence fewer independent sets), so $M_d(n)$ is a principal submatrix of $\widehat M_d(n)$. For symmetric nonnegative matrices, restricting to a principal submatrix cannot increase the spectral radius (extend the submatrix's Perron eigenvector by zero in the Rayleigh quotient $\rho(A)=\max_{x\ge0,x\ne0}\left(\frac{x^{\intercal}Ax}{x^\intercal x}\right)$), so
\[
\lambda_d(n)\ \le\ \eta_d(n)\ \le\ \lambda_d(n+1).
\]
Dividing logarithms by $n^{d-1}$ and using Proposition~\ref{prop:Perron_Frob} ($\lambda_d(n)^{\frac{1}{n^{d-1}}}\to\kappa_d$, so also $\lambda_d(n+1)^{\frac{1}{n^{d-1}}}=\big(\lambda_d(n+1)^{\frac{1}{(n+1)^{d-1}}}\big)^{\frac{(n+1)^{d-1}}{n^{d-1}}}\to\kappa_d$, as $\frac{(n+1)^{d-1}}{n^{d-1}}\to1$) sandwiches $\eta_d(n)^{\frac{1}{n^{d-1}}}$ between two sequences converging to $\kappa_d$.
\end{proof}

Since $\widehat M_d(n)$ is a fixed integer matrix for each $n$, Proposition~\ref{prop:beta} attaches to it a canonical $p$-adic companion, exactly as it did for $M_d(p^k)$:

\begin{proposition}\label{prop:betaM}
For every $d\ge2$, for every $n\ge2$, and for every prime $p$, we have that  $\beta(\widehat M_d(n))=\lim_j^{(p)}\tr(\widehat M_d(n)^{p^j})$ exists in $\Zp$, and equals the sum of Teichm\"uller lifts of the eigenvalues of $\widehat M_d(n)\bmod p$.
\end{proposition}

So every finite-size approximant to $\kappa_d$ - whether built from tori ($\kappa_d^{(p)}$, via $M_d(p^k)$) or from strips ($\beta(\widehat M_d(n))$, via Pavlov's $\widehat M_d(n)$) - carries the same kind of algebraic $p$-adic analogoue.

\begin{remark}
It is not clear whether $\beta(\widehat M_d(n))$ converge $p$-adically as $n\to\infty$. Pavlov's sharpened, exponentially fast archimedean convergence $h_{n+1}-h_n\to\log\kappa_2$ \cite{Pavlov2012}, which was proven for $d=2$ by percolation-theoretic and ergodic-theoretic methods well beyond the elementary combinatorics we use here,  would, if it held $p$-adically too, give a strip-based construction of a $p$-adic limit for $\kappa_d$ genuinely independent of $\kappa_d^{(p)}$.
\end{remark}

We now turn to proving an approximation theorem about $\kappa_d$ that provides for us two algebraicity conditions. Before doing that, we recall some basic properties of the absolute logarithmic Weil height $h$ on $\overline\Q$ (see e.g.\ \cite{BombieriGubler,WaldschmidtDA} for proofs):
\begin{definition}
    Let $\gamma$ be a non-zero algebraic number of degree $g=[\Q(\gamma):\Q]$. Let  $\gamma=\gamma_1,\dots,\gamma_g$ be the conjugates of $\gamma$. Let $a_0X^g+\cdots+a_g\in\Z[X]$ be the  a minimal polynomial of $\gamma$ such that $a_0>0$ and $\gcd(a_0, \dots, a_g)=1$.
    \begin{enumerate}
        \item The \textbf{Mahler measure} of $\gamma$ is defined to be $M(\gamma)=a_0\prod_{i=1}^g\max(1,|\gamma_i|)$. 
        \item The \textbf{logarithmic Weil height} of $\gamma$ is defined to be $h(\gamma)=\tfrac1g\log M(\gamma)$. 
    \end{enumerate}
\end{definition}

\begin{remark}
    Observe that $h(\gamma) \geq 0$ for every $\gamma$ and that $h(\gamma_1\pm\gamma_2)\le h(\gamma_1)+h(\gamma_2)+\log2$ for every $\gamma_1$ and $\gamma_2$. 
\end{remark}
%for a nonzero algebraic number $\gamma$ of degree $\delta=$ with ,  and . We use: $h(m)=\log\max(1,|m|)$ for $m\in\Z$, $h(\alpha^k)=|k|\,h(\alpha)$ for $k\in\Z$, and the subadditivity .

\begin{lemma}\label{lem:liouville}
For every nonzero algebraic number $\gamma$ we have that $-\log|\gamma|\ \le\ [\Q(\gamma):\Q]\cdot h(\gamma)$.
\end{lemma}
\begin{proof}
With notation as above, the constant term of the minimal polynomial satisfies $a_0\prod_i\gamma_i=(-1)^\delta a_\delta$, and $a_\delta\ne0$ (as $\gamma\ne0$ is a root of the irreducible polynomial, which is therefore not divisible by $X$), so $|a_\delta|\ge1$. Hence
\[
1 \le |a_\delta| = a_0\prod_{i=1}^\delta|\gamma_i| = a_0|\gamma_1|\prod_{i\ge2}|\gamma_i| \le a_0|\gamma_1|\prod_{i\ge2}\max(1,|\gamma_i|) \le |\gamma_1|\cdot a_0\prod_{i=1}^\delta\max(1,|\gamma_i|) = |\gamma_1|\cdot M(\gamma),
\]
using $\max(1,|\gamma_1|)\ge1$ to insert the missing factor. So $|\gamma|=|\gamma_1|\ge M(\gamma)^{-1}=e^{-\delta h(\gamma)}$, i.e.\ $-\log|\gamma|\le\delta h(\gamma)$.
\end{proof}

\begin{lemma}\label{lem:betaheight}
For $n\ge1$, write $\gamma_n=a_d(n)^{\frac{1}{n^d}}$ (the positive real $n^d$-th root). Then:
\begin{enumerate}
    \item $\gamma_n\in(1,2]$,
    \item $\gamma_n$ is an algebraic integer all of whose conjugates have absolute value exactly $\gamma_n$, 
    \item $h(\gamma_n)=\log\gamma_n\in(0,\log2]$.
\end{enumerate}
\end{lemma}
\begin{proof}
Since $\emptyset$ and every singleton are independent and $n^d\ge1$, we have that $a_d(n)\ge2$, and so $\gamma_n>1$. In addition, from Lemma~\ref{lem:naive_bound} we have that $a_d(n)\le2^{n^d}$, and so $\gamma_n\le2$. By definition $\gamma_n$ is a root of the monic integer polynomial $X^{n^d}-a_d(n)$, hence an algebraic integer, every root of this polynomial has the form $\gamma_n\cdot\xi$ for an $n^d$-th root of unity $\xi$, hence modulus exactly $\gamma_n$, and in particular every conjugate of $\gamma_n$ (being among these roots) has modulus $\gamma_n$. So, with $\delta=\deg(\gamma_n)$, we have that $M(\gamma_n)=\prod_{i=1}^\delta\max(1,\gamma_n)=\gamma_n^\delta$ (as $\gamma_n>1$, and $\gamma_n$ is an algebraic integer so $a_0=1$), giving $h(\gamma_n)=\tfrac1\delta\log(\gamma_n^\delta)=\log\gamma_n\in(0,\log2]$.
\end{proof}

\begin{remark}
Recall that Northcott's theorem says that algebraic numbers of bounded degree and bounded logarithmic Weil height form a finite set (see~\cite{Northcott1949}). In particular, if $\kappa_d$ is algebraic for infinitely many values of $d$, then along those dimensions, the degrees and logarithmic heights cannot both remain bounded. This is true since from  Theorem~\ref{thm:decreasing} we have that $\kappa_d\to\sqrt2$ with $\kappa_d>\sqrt2$ for every finite $d$, but a sequence whose values are in a finite set cannot converge to $\sqrt2$ without eventually equaling $\sqrt2$, which is impossible.
\end{remark}

We are now ready to prove our approximation theorem:

\begin{theorem}\label{thm:master}
Suppose $\kappa_d$ is algebraic with $D=[\Q(\kappa_d):\Q]$ and $h_0=h(\kappa_d)$. Then:
\begin{enumerate}
\item  \textup{(Archimedean gap bound.)} For every $n\ge1$ with $a_d(n)\ne\kappa_d^{n^d}$,
\[
-\frac1{n^d}\log\big|a_d(n)-\kappa_d^{n^d}\big|\ \le\ D(\log2+h_0)+\frac{D\log2}{n^d}.
\]
\item \textup{(Quantitative degree growth.)} Set $C_1=h_0+2\log2$. For every $n$ such that $a_d(n)\ne\kappa_d^{n^d}$ we have that,
\[
\deg(\gamma_n)\ \ge\ \frac{-\log|\gamma_n-\kappa_d|}{D\,C_1} \ge\ \frac{\log n - \log(2\sqrt3\,d\log2)}{D\,C_1}
\]
\end{enumerate}
\end{theorem}
\begin{proof}
For part 1, set $\alpha_n=a_d(n)-\kappa_d^{n^d}\in\Q(\kappa_d)$, nonzero by hypothesis, so $[\Q(\alpha_n):\Q]\le D$. By Lemma~\ref{lem:liouville} (using $[\Q(\alpha_n):\Q]\le D$ and $h(\alpha_n)\ge0$, so the Liouville bound only improves upon enlarging the degree factor to $D$):
\[
-\log|\alpha_n|\ \le\ D\cdot h(\alpha_n)\ \le\ D\big(h(a_d(n))+h(\kappa_d^{n^d})+\log2\big)\ =\ D\big(\log a_d(n)+n^dh_0+\log2\big),
\]
using $h(a_d(n))=\log a_d(n)$ (as $a_d(n)$ is a positive integer) and $h(\kappa_d^{n^d})=n^dh_0$. Since $a_d(n)\le2^{n^d}$ (Lemma~\ref{lem:betaheight}'s proof), $\log a_d(n)\le n^d\log2$, so $-\log|\alpha_n|\le D\big(n^d\log2+n^dh_0+\log2\big)=Dn^d(\log2+h_0)+D\log2$, and dividing by $n^d$ gives the stated bound.\\ % if the left side tended to $\infty$ this would contradict the bound (whose right side tends to the finite constant $D(\log2+h_0)$), so no algebraic $\kappa_d$ can satisfy the stated hypothesis, i.e.\ $\kappa_d$ must be transcendental.\\

For part 2, set $\delta_n=\gamma_n-\kappa_d\ne0$ (as $a_d(n)\ne\kappa_d^{n^d}$), lying in the compositum $\Q(\gamma_n,\kappa_d)$, of degree $[\Q(\delta_n):\Q]\le\deg(\gamma_n)\cdot D$. By Lemma~\ref{lem:liouville} (again enlarging the degree factor, valid as $h(\delta_n)\ge0$) and Lemma~\ref{lem:betaheight} ($h(\gamma_n)=\log\gamma_n\le\log2$), we can conclude that
\[
-\log|\delta_n|\ \le\ \deg(\gamma_n)\cdot D\cdot h(\delta_n)\ \le\ \deg(\gamma_n)\cdot D\cdot\big(h(\gamma_n)+h(\kappa_d)+\log2\big)\ \le\ \deg(\gamma_n)\cdot D\cdot(\log2+h_0+\log2),
\]
This is $\deg(\gamma_n)\cdot D\cdot C_1$, giving $\deg(\gamma_n)\ge\dfrac{-\log|\delta_n|}{DC_1}$, the first claim. Since from Theorem~\ref{thm:existence} we have that $\gamma_n\to\kappa_d$ and that $\delta_n\ne0$ for $a_d(n)\ne\kappa_d^{n^d}$, we can conclude that $-\log|\delta_n|\to\infty$ as $n\to\infty$ through $n$ such that $a_d(n)\ne\kappa_d^{n^d}$, forcing $\deg(\gamma_n)\to\infty$. Now, from $|\log\gamma_n-\log\kappa_d|\le \frac{d\log2}{n}$, for $n$ large enough that $\frac{d\log2}{n}\le1$, write $t=\log\gamma_n-\log\kappa_d$, then we have that $|t|\le \frac{d\log2}{n}\le1$, and since $|e^t-1|\le2|t|$ for $|t|\le1$ (as for $t\in[0,1]$, the function $\varphi(t)=e^t-1-2t$ satisfies $\varphi(0)=0$ and $\frac{d\varphi}{dt}=e^t-2$, which is negative on $[0,\log2)$ and positive on $(\log2,1]$, while $\varphi(1)=e-3<0$, since $\varphi$ decreases away from $0$ and only turns back upward while remaining below $\varphi(1)<0$, we have that $\varphi\le0$ throughout $[0,1]$, i.e.\ $e^t-1\le2t$ there, for $t\in[-1,0]$, convexity of $e^t$ gives $e^t\ge1+t$, so $1-e^t\le-t\le2|t|$, combining both ranges gives $|e^t-1|\le2|t|$ on $[-1,1]$), we get
\[
|\delta_n|\ =\ |\gamma_n-\kappa_d|\ =\ \kappa_d|e^t-1|\ \le\ \sqrt3\cdot2\cdot\frac{d\log2}n\ =\ \frac{2\sqrt3\,d\log2}n,
\]
using $\kappa_d\le\sqrt3$ (Lemma~\ref{lem:naive_bound}). So $-\log|\delta_n|\ge\log n-\log(2\sqrt3\,d\log2)$, which gives the desired result together with the fact that $|\log\gamma_n-\log\kappa_d|\le d\log2/n$, which follows from Theorem~\ref{thm:existence}.
\end{proof}

From Theorem~\ref{thm:master} we can conclude two criteria for the algebraicity of $\kappa_d$:

\begin{corollary}
Assume that $a_d(n)\ne\kappa_d^{n^d}$ for every $n$. If either
\begin{enumerate}
    \item $-\tfrac1{n^d}\log|a_d(n)-\kappa_d^{n^d}|\to\infty$ or
    \item $\{\deg(\gamma_n)\}$ is a bounded sequence,
\end{enumerate}
$\kappa_d$ must be transcendental. 
\end{corollary}

\begin{remark}
Both parts of Theorem~\ref{thm:master} are instances of a single mechanism, namely, bounding via the Liouville inequality of Lemma~\ref{lem:liouville}, how well a fixed algebraic number $\kappa_d$ can be approximated by the specific algebraic numbers $\gamma_n=a_d(n)^{1/n^d}$ that the combinatorics of $T_d(n)$ actually produces. Both derivations use only Lemma~\ref{lem:betaheight}'s uniform height bound on $\gamma_n$ and the elementary bound $a_d(n)\le2^{n^d}$. 
\end{remark}

%\subsection*{A local-field valuation lemma, and the limits of $\kappa_d^{p^{kd}}$}

We end this section with a study of the valuations of $\kappa_d$ with respect to the local field $\Q(\kappa_d)$ (under the assumption that $\kappa_d$ is algebraic). Throughout the rest of this subsection, $\kappa_d$ is algebraic, $K=\Q(\kappa_d)$, and $\mathfrak p\mid p$ is a prime of $K$ with normalized valuation $v_\mathfrak p$,  and $e_\mathfrak p=v_\mathfrak p(p)$.  %(Throughout this and the following subsections, $\log$ denotes the $p$-adic logarithm on the relevant local field, as opposed to the archimedean $\log$ of the preceding subsection, the type of the argument - an element of $L^\times$ versus a positive real number - makes the intended meaning clear from context.)

\begin{lemma}\label{lem:local}
Let $L/\Qp$ be a finite extension with normalized valuation $v$ and $e=v(p)$. If $u\in L^\times$ has $v(u-1)>0$ and $u$ is not a root of unity, there exist $C\in\Z$ and $m_0\ge0$ with $v(u^{p^m}-1)=C+me$ for every $m\ge m_0$.
\end{lemma}
\begin{proof}
 Let $t_0=v(u-1)>0$ and $x=u-1$. Then we can write  $u^p-1=(1+x)^p-1=\sum_{j=1}^p\binom pjx^j$. For $1\le j\le p-1$ we have that  $\nu_p\big(\binom pj\big)=1$ exactly, so $v\big(\binom pjx^j\big)=e+jt_0\ge e+t_0$, and $v(x^p)=pt_0$. Thus, since $e\ge1$ and $t_0>0$, we can conclude that  $v(u^p-1)\ge\min(e+t_0,pt_0)=t_0+\min(e,(p-1)t_0)>t_0$. Since $u$ is not a root of unity, neither is $u^{p^j}$ for any $j$, so this computation reapplies at every stage: $v(u^{p^j}-1)$ is strictly increasing in $j$, hence $\to\infty$. Now, choose $s$ such that $v(u^{p^s}-1)>e/(p-1)$. On the range $\{x\in L:v(x-1)>e/(p-1)\}$, the $p$-adic logarithm converges with $v(\log_p x)=v(x-1)$ (see e.g.\ \cite{Serre,Neukirch}), and this range is closed under $p$-th powers: if $v(w-1)=t>e/(p-1)$, then $(p-1)t>e$, so $v(w^p-1)\ge t+\min(e,(p-1)t)=t+e>e/(p-1)$. Setting $w=u^{p^s}$ we have that $\log_p(w^{p^j})=p^j\log_p w$ for $j\ge0$, so $v(\log_p(w^{p^j}))=v(\log_p w)+je$. Since $w^{p^j}$ stays in the range for every $j\ge0$, we can conclude that $v(w^{p^j}-1)=v(\log_p(w^{p^j}))=v(w-1)+je$. Therefore, since $w^{p^j}=u^{p^{s+j}}$, we can set $m=s+j$ and conclude that $v(u^{p^m}-1)=v(w-1)+(m-s)e=:C+me$ for $m\ge m_0=s$.
\end{proof}

\begin{remark}
    We can view Lemma~\ref{lem:local} as a local-field valuation estimate that tracks how fast $u^{p^m}$ converges to $1$ for some $u\equiv1\pmod{\mathfrak p}$ in a finite extension $L/\Qp$.
\end{remark}

\begin{proposition}\label{prop:teich}
If $v_\mathfrak p(\kappa_d)=0$ and $\mathfrak p$ has residue degree $1$, write $\kappa_d=\omega(\kappa_d)\langle\kappa_d\rangle$ (Teichm\"uller decomposition). Then $\kappa_d^{p^{km}}\to\omega(\kappa_d)$ as $k\to\infty$, for every fixed $m\ge1$ (in particular for $m=d$).
\end{proposition}
\begin{proof}
$p^{km}\equiv1\pmod{p-1}$ trivially, so $\omega(\kappa_d)^{p^{km}}=\omega(\kappa_d)$ exactly for every $k$. For $\langle\kappa_d\rangle$: if it equals $1$, the claim is trivial. If it is a root of unity $\ne1$, it must have $p$-power order (a root of unity of order coprime to $p$ reduces injectively into the residue field, so cannot lie in $1+\mathfrak p$ unless it is $1$), so $\langle\kappa_d\rangle^{p^{km}}=1$ once $km \geq s$. If $\langle\kappa_d\rangle$ is not a root of unity, Lemma~\ref{lem:local} gives $v_\mathfrak p\big(\langle\kappa_d\rangle^{p^{km}}-1\big)\to\infty$ as $k\to\infty$. Either way $\langle\kappa_d\rangle^{p^{km}}\to1$, so $\kappa_d^{p^{km}}=\omega(\kappa_d)\langle\kappa_d\rangle^{p^{km}}\to\omega(\kappa_d)$.
\end{proof}

\begin{lemma}\label{lem:dichotomy}
Fix $\kappa_d$ algebraic and let $\mathfrak p\mid p$. Then exactly one of the following holds:
\begin{enumerate}
\item $v_\mathfrak p(\kappa_d-1)>0$, and then $v_\mathfrak p(\kappa_d^{p^{kd}}-1)\to\infty$ as $k\to\infty$ (rate given by Lemma~\ref{lem:local}),
\item $v_\mathfrak p(\kappa_d)<0$, and then $v_\mathfrak p(\kappa_d^{p^{kd}}-1)=p^{kd}v_\mathfrak p(\kappa_d)\to-\infty$,
\item $v_\mathfrak p(\kappa_d)\ge0$ and $\kappa_d\not\equiv1\pmod{\mathfrak p}$, and then $v_\mathfrak p(\kappa_d^{p^{kd}}-1)=0$ for every $k\ge0$.
\end{enumerate}
\end{lemma}
\begin{proof}
These three cases exhaust all possibilities for $v_\mathfrak p(\kappa_d)$ and, when $\ge0$, whether $\kappa_d\equiv1\pmod{\mathfrak p}$ or not. For the second case,  $v_\mathfrak p(\kappa_d^{p^{kd}})=p^{kd}v_\mathfrak p(\kappa_d)<0=v_\mathfrak p(1)$, so the ultrametric equality gives $v_\mathfrak p(\kappa_d^{p^{kd}}-1)=v_\mathfrak p(\kappa_d^{p^{kd}})$. For the third case, $\kappa_d\in\mathcal O_{K_\mathfrak p}$, reducing to $\bar\kappa_d\in F=\mathcal O/\mathfrak p$. If $\bar\kappa_d=0$, then $\bar\kappa_d^{p^{kd}}=0\ne1$ for every $k$. If $\bar\kappa_d\ne0$, its order $\ell\ge2$ (as $\bar\kappa_d\ne1$) divides $|F^\times|=p^{f_\mathfrak p}-1$, coprime to $p$, so $\gcd(\ell,p^{kd})=1$, and therefore $\bar\kappa_d^{p^{kd}}=1$ would force $\ell\mid p^{kd}$, hence $\ell=1$, contradiction. Either way, $\bar\kappa_d^{p^{kd}}\ne1$, so $v_\mathfrak p(\kappa_d^{p^{kd}}-1)\le0$, since $\kappa_d^{p^{kd}}$ and $1$ are both $\mathfrak p$-integral, $v_\mathfrak p(\kappa_d^{p^{kd}}-1)\ge0$ too, giving exactly $0$. Finally, the first case follows from Lemma~\ref{lem:local}.
\end{proof}

\begin{proposition}\label{prop:ANppinned}
For every $\mathfrak p\mid p$ of $K=\Q(\kappa_d)$, if $p\in\Tame_d$, then $v_\mathfrak p\big(\kappa_d^{(p)}-1\big)=e_\mathfrak p\,r_p$, where $r_p=\nu_p(a_d(p)-1)$.
\end{proposition}
\begin{proof}
By Proposition~\ref{prop:kappa_p_exists}, $\nu_p(\kappa_d^{(p)}-1)=r_p$ in $\Zp$, extending to $K_\mathfrak p$ via $v_\mathfrak p(x)=e_\mathfrak p\nu_p(x)$ for $x\in\Qp^\times\subset K_\mathfrak p^\times$ gives the claim.
\end{proof}

\begin{theorem}\label{thm:comparison}
Assume $\kappa_d$ algebraic, $p\in\Tame_d$, and fix $\mathfrak p\mid p$.
\begin{enumerate}
\item If $v_\mathfrak p(\kappa_d-1)>0$ then for all large $k$ we have that $v_\mathfrak p\big(\kappa_d^{(p)}-\kappa_d^{p^{kd}}\big)=e_\mathfrak p\,r_p$.
\item If $v_\mathfrak p(\kappa_d-1)\le0$ then for every $k$ we have that $v_\mathfrak p\big(\kappa_d^{(p)}-\kappa_d^{p^{kd}}\big)=v_\mathfrak p\big(\kappa_d^{p^{kd}}-1\big)\le0$.
\end{enumerate}
In particular, $\kappa_d^{(p)}\ne\kappa_d^{p^{kd}}$ for all large $k$ and  $\kappa_d^{(p)}$ is never a $\mathfrak p$-adic limit point of the sequence $\big\{\kappa_d^{p^{kd}}\big\}_{k=1}^\infty$.
\end{theorem}
\begin{proof}
Write $x=\kappa_d^{(p)}-1$ (of fixed valuation $e_\mathfrak p r_p$, as in Proposition~\ref{prop:ANppinned}) and $y=\kappa_d^{p^{kd}}-1$. In the first case we have that $v_\mathfrak p(y)\to\infty$ (Lemma~\ref{lem:dichotomy}), so eventually exceeds $v_\mathfrak p(x)=e_\mathfrak pr_p$, giving $v_\mathfrak p(x-y)=v_\mathfrak p(x)$. In the second case we have that $v_\mathfrak p(y)\le0<e_\mathfrak p r_p=v_\mathfrak p(x)$ for every $k$ (Lemma~\ref{lem:dichotomy}, cases 2 or 3), giving us $v_\mathfrak p(x-y)=v_\mathfrak p(y)$.
\end{proof}

\section{The Equation $a_d(n)=\kappa_d^{n^d}$ and the set  $E_d$}\label{sec:Ed}

As we saw in Theorem~\ref{thm:master}, the inequality $a_d(n)\ne\kappa_d^{n^d}$ plays a crucial role in the analysis of the algebraicity of $\kappa_d$. Therefore, in this section we study for which values of $n$ we can have that $a_d(n)= \kappa_d^{n^d}$ and how many such values there can be. Throughout this section we assume that $\kappa_d$ is algebraic, writing $D=[K:\Q]$ where $K=\Q(\kappa_d)$, as before.

\begin{definition}\label{def:Ed}
$E_d=\{n\ge1 : a_d(n)=\kappa_d^{n^d}\}=\{n\ge1:\gamma_n=\kappa_d\}$, where $\gamma_n=a_d(n)^{1/n^d}$ as in Lemma~\ref{lem:betaheight}.
\end{definition}

\begin{remark}\label{lem:1notin}
Since from Lemma~\ref{lem:naive_bound} we have that $\kappa_d \in [\sqrt{2}, \sqrt{3}]$, then $\kappa_d$ is not an integer and so $1\notin E_d$ (noting that $a_d(1)=2$).
\end{remark}

\noindent  The following proposition tells us that $E_d$ must have a strict arithmetic structure:

\begin{proposition}\label{thm:coprimality}
If $\kappa_d$ is algebraic and $E_d\ne\emptyset$, then $\gcd(E_d)>1$, i.e.\ some fixed prime divides every element of $E_d$.
\end{proposition}
\begin{proof}
Suppose $\gcd(E_d)=1$. Since the sequence of gcds of the initial elements of any enumeration of $E_d$ is a non-increasing sequence of positive integers, it stabilizes, so there is a finite subset $\{n_1,\dots,n_r\}\subseteq E_d$ with $\gcd(n_1,\dots,n_r)=1$, hence also $\gcd(n_1^d,\dots,n_r^d)=\gcd(n_1,\dots,n_r)^d=1$. By Bézout, there are integers $u_1,\dots,u_r$ with $\sum_iu_in_i^d=1$. Since $n_i\in E_d$, we have that $\kappa_d^{n_i^d}=a_d(n_i)\in\Z_{\ge2}$ for each $i$, and  so
\[
\kappa_d\ =\ \kappa_d^{\sum_iu_in_i^d}\ =\ \prod_i\big(\kappa_d^{n_i^d}\big)^{u_i}\ =\ \prod_i a_d(n_i)^{u_i}\ \in\ \Q,
\]
a positive rational number (as each $a_d(n_i)$ is a positive integer). But $\kappa_d$ is a root of the monic integer polynomial $X^{n_1^d}-a_d(n_1)$, hence an algebraic integer, a rational algebraic integer is an ordinary integer, so $\kappa_d\in\Z$. This contradicts $\kappa_d\in[\sqrt2,\sqrt3]$ (Lemma~\ref{lem:naive_bound}), an interval containing no integers.
\end{proof}

From Theorem~\ref{thm:coprimality} we can conclude that $E_d$ has a very limited arithmetic structure

\begin{corollary}\label{cor:pairwise}
If $\kappa_d$ is algebraic and $E_d \neq \emptyset$, then any two elements of $E_d$ share a common factor $>1$. In particular, if $p^k \in E_d$ for some prime $p$ and some $k \geq 1$ then $q^l \notin E_d$ for any prime $q \neq p$ and any $l \geq 1$. 
\end{corollary}

% \begin{corollary}\label{cor:oneprime}
% If $\kappa_d$ is algebraic and $E_d\ne\emptyset$, then $E_d$ contains at most one prime number.
% \end{corollary}
% \begin{proof}
% By , fix a prime $p_0\mid\gcd(E_d)$, so $p_0\mid n$ for every $n\in E_d$. If a prime $q\in E_d$, then $p_0\mid q$ forces $p_0=q$ (as $q$ is prime and $p_0>1$). So the only prime that can possibly lie in $E_d$ is $p_0$ itself.
% \end{proof}

We now show that, moreover, only finitely many powers of a single tame prime can lie in $E_d$, giving a genuine finiteness statement for prime powers:

\begin{proposition}\label{thm:finiteexp}
Suppose $\kappa_d$ is algebraic, $p$ is a tame prime for $d$, and set $m=\nu_p(a_d(p)-1)$. Then we have that:
\begin{enumerate}
\item If $p$ is odd then $\{k\ge1 : p^k\in E_d\}$ is finite and of size at most $m$.
\item If $p=2$ then $\{k\ge1 : 2^k\in E_d\}$ has at most one element.
\end{enumerate}
\end{proposition}
\begin{proof}
Let $k_1<k_2$ both satisfy $p^{k_1},p^{k_2}\in E_d$, and set $A=a_d(p^{k_1})\ge2$, by Proposition~\ref{prop:tame_iff} (tameness of $p$), $\nu_p(A-1)=\nu_p(a_d(p)-1)=m$. Since $p^{k_1},p^{k_2}\in E_d$, we have that $A=\kappa_d^{p^{k_1d}}$ and that $a_d(p^{k_2})=\kappa_d^{p^{k_2d}}=\big(\kappa_d^{p^{k_1d}}\big)^{p^{(k_2-k_1)d}}=A^{p^{(k_2-k_1)d}}$. Set $r=(k_2-k_1)d\ge1$.  By Theorem~\ref{prop:Gauss}, $\nu_p(a_d(p^j)-a_d(p^{j-1}))\ge j$ for every $j\ge1$, telescoping from $j=k_1+1$ to $k_2$,
\[
a_d(p^{k_2})-A\ =\ \sum_{j=k_1+1}^{k_2}\big(a_d(p^j)-a_d(p^{j-1})\big),
\]
a sum of terms each of valuation $\ge k_1+1$, so $\nu_p(a_d(p^{k_2})-A)\ge k_1+1$. Combined with Step 1, $\nu_p(A^{p^r}-A)\ge k_1+1$. We claim $\nu_p(A^{p^r}-A)=m$ for every $r\ge1$.\\

First, assume that $p$ is an odd prime. write $A=1+p^mu$ where $p\nmid u$. A direct binomial expansion of $(1+p^mu)^p=1+p^{m+1}u+\sum_{i=2}^p\binom pip^{im}u^i$ shows every term with $i\ge2$ has $p$-valuation $\ge1+im\ge1+2m>m+1$ (as $p\mid\binom pi$ for $1\le i\le p-1$ contributing valuation $\ge1$, and the $i=p$ term has valuation $pm\ge3m>m+1$ for $m\ge1,p\ge3$), as the $i=1$ term has valuation exactly $m+1$, the ultrametric inequality gives $\nu_p(A^p-1)=m+1$ exactly. Iterating (the same computation applies verbatim with $m$ replaced by $m+1,m+2,\dots$ at each step, since $m+j\ge1$ throughout), $\nu_p(A^{p^r}-1)=m+r$ for every $r\ge0$. Since $r\ge1$, we can conclude that $m+r>m=\nu_p(A-1)$, and so by the ultrametric inequality (valuations differ, hence valuation of the difference equals the smaller one):
\[
\nu_p(A^{p^r}-A)\ =\ \nu_p\big((A^{p^r}-1)-(A-1)\big)\ =\ \min(m+r,\,m)\ =\ m.
\]

 Now, assume that $p=2$. Here $m=1$ (Corollary~\ref{thm:p2}), so $A\equiv3\pmod4$. If we write $A=1+2u$ for some odd $u$, then $A+1=2+2u=2(1+u)$ with $1+u$ even (as $u$ is odd), so $\nu_2(A+1)=:e\ge2$. The $p=2$ lifting-the-exponent identity (a standard fact, verified directly by induction from $A^{2^{r}}-1=(A^{2^{r-1}}-1)(A^{2^{r-1}}+1)$ and tracking that $A^{2^{r-1}}+1\equiv2\pmod4$ for $r\ge2$) gives $\nu_2(A^{2^r}-1)=\nu_2(A-1)+\nu_2(A+1)+r-1=1+e+r-1=e+r$ for $r\ge1$. As $e\ge2$ and $r\ge1$, we get that $e+r\ge3>1=m$, and so again by the ultrametric inequality, $\nu_2(A^{2^r}-A)=\min(e+r,1)=1=m$.\\

Therefore, in both cases we have that $k_1+1\le\nu_p(A^{p^r}-A)=m$, i.e.\ $k_1\le m-1$. Thus, we can conclude that any two elements $k_1<k_2$ of $\{k:p^k\in E_d\}$, necessarily $k_1\le m-1$. In particular at most one element of this set can exceed $m-1$ (if $k_1,k_2$ both exceeded $m-1$ with $k_1<k_2$, then $k_1\le m-1$, a contradiction), so the set has size at most $(m-1)+1=m$. This proves part 1. For part 2 (i.e. $p=2$ and $m=1$), the bound gives size $\le1$.
\end{proof}

We can in fact strengthen Proposition~\ref{thm:finiteexp} by looking at the degree of $\kappa_d$ (as an algebraic number). For that, we need the following lemma:

\begin{lemma}\label{lem:radical-degree}
Let $A\ge2$ be an integer, $p$ a prime, and $r\ge1$. Let $s$ be the largest integer with $0\le s\le r$ such that $A$ is a $p^s$-th power in $\Z$. Then $[\Q(A^{1/p^r}):\Q]=p^{r-s}$
\end{lemma}

\begin{proof}
Write $A=B^{p^s}$. If $s=r$, the radical is the integer $B$ and the formula is clear. Suppose $s<r$. Then $B$ is not a $p$-th power, so some rational prime $\ell$ satisfies $p\nmid\nu_\ell(B)$.
Set $\alpha=B^{1/p^{r-s}}$ and set $L=\Q(\alpha)$. For a prime $\mathfrak L$ of $L$ above $\ell$, normalize $v_{\mathfrak L}$ by $v_{\mathfrak L}(\ell)=e_{\mathfrak L/\ell}$. From $\alpha^{p^{r-s}}=B$ we get that $p^{r-s}v_{\mathfrak L}(\alpha)
=e_{\mathfrak L/\ell}\nu_\ell(B)$.
Since $p\nmid\nu_\ell(B)$, we get that  $p^{r-s}\mid e_{\mathfrak L/\ell}$. Thus we can conclude that $[L:\Q]\ge e_{\mathfrak L/\ell}\ge p^{r-s}$. Yet,  $\alpha$ satisfies $x^{p^{r-s}}-B=0$, therefore $[L:\Q]\le p^{r-s}$ and the result follows.
\end{proof}

\begin{corollary}\label{cor:degree-23}
For every $d,k\ge1$ we have that $[\Q(\gamma_{2^k}):\Q]=2^{kd}$ and that 
and $[\Q(\gamma_{3^k}):\Q]=3^{kd}$.
\end{corollary}

\begin{proof}
From Corollary~\ref{thm:p2} we have that $p=2$ and $p=3$ are tame for every $d$.In addition, since $a_d(2^k)\equiv3\pmod4$, we have that $a_d(2^k)$ is not a square. In addition, since $\nu_3(a_d(3^k)-1)=1$, if $a_d(3^k)=B^3$, then $B\equiv1\pmod3$, and so we would have that $\nu_3(B^3-1)=\nu_3(B-1)+1\ge2$ which is a contradiction. Thus $a_d(3^k)$ is not a cube, and we can apply Lemma~\ref{lem:radical-degree} with $r=kd$.
\end{proof}

\begin{corollary}\label{cor:tame-degree}
If $p\in \Tame_d$ is odd, then $[\Q(\gamma_{p^k}):\Q]
\ge p^{kd+1-\nu_p(a_d(p)-1)}$. If, in addition, $\nu_p(a_d(p)-1)=1$, then the degree is maximal, i.e. $[\Q(\gamma_{p^k}):\Q]=p^{kd}.$
\end{corollary}

\begin{proof}
Suppose $a_d(p^k)=B^{p^s}$. Since $a_d(p^k)\equiv1\pmod p$, then we have that $B\equiv1\pmod p$. Since $p$ is odd, we have that $\nu_p(B^{p^s}-1)=\nu_p(B-1)+s\ge s+1$, and since $p \in \Tame_d$ we have that $\nu_p(B^{p^s}-1)= \nu_p(a_d(p)-1)$, so $s\le \nu_p(a_d(p)-1)-1$, and the result follows from Lemma~\ref{lem:radical-degree}.
\end{proof}

\begin{theorem}\label{thm:radical-degree}
Assume $\kappa_d$ is algebraic with $D=[\Q(\kappa_d):\Q]$ and suppose $p^k\in E_d$. Then:
\begin{enumerate}
\item $D$ is a positive power of $p$.
\item If $p \in \Tame_d$ is odd, then $D=p^r$ for $r\le kd\le r+\nu_p(a_d(p)-1)-1$. In particular, 
\[
|\{k:p^k\in E_d\}|
\le
\left\lfloor\frac{\nu_p(a_d(p)-1)-1}{d}\right\rfloor+1.
\]
\item If $p=2$ or $p=3$, then $D=p^{kd}$. In particular, at most one power of $p$ can belong to $E_d$.
\end{enumerate}
\end{theorem}

\begin{proof}
Since $\kappa_d^{p^{kd}}=a_d(p^k)$ we have that $\kappa_d$ is the positive $p^{kd}$-th root of the integer $a_d(p^k)$. From Lemma~\ref{lem:radical-degree} we get that $D$ is a power of $p$. In the tame odd-prime case, Corollary~\ref{cor:tame-degree} gives . Counting the integer values of $k$ in an interval of length $m-1$ after scaling by $d$ gives the stated cardinality bound. The cases $p=2,3$ follow from \ref{cor:degree-23}.
\end{proof}

\begin{corollary}
    If $\kappa_d$ is algebraic and $[\Q(\kappa_d) \colon \Q]$ is not a prime power then $E_d$ contains no prime power.
\end{corollary}

\begin{proof}
    This follows directly from Theorem~\ref{thm:radical-degree}.
\end{proof}

\begin{theorem}\label{thm:localrigidity}
Suppose $\kappa_d$ is algebraic, let $K=\Q(\kappa_d)$, and assume that $p^k\in E_d$ for some prime $p$ and $k\ge1$. Then:
\begin{enumerate}
\item $\kappa_d$ is $\mathfrak p$-integral (i.e.\ $v_{\mathfrak p}(\kappa_d)\ge0$) at every prime $\mathfrak p$ of $K$ above $p$,
\item $\kappa_d\equiv1\pmod{\mathfrak p}$ for every such $\mathfrak p$. Hence, the reduction of $\kappa_d$ in the residue field $k_\mathfrak{p}$ of  $\mathcal{O}_K$ at $\mathfrak{p}$ agrees with the reduction of $\kappa_d^{(p)}\in1+p\Zp$.
\end{enumerate}
\end{theorem}
\begin{proof}
For part 1, since $p^k\in E_d$ then from Lemma~\ref{lem:betaheight} we have that $\kappa_d^{p^{kd}}=a_d(p^k)$ is an ordinary positive integer, hence $\mathfrak p$-integral (with $v_{\mathfrak p}\ge0$) at every prime $\mathfrak p$ of $K$, for every prime of $\mathbb{Q}$. So $p^{kd}\cdot v_{\mathfrak p}(\kappa_d)=v_{\mathfrak p}\big(\kappa_d^{p^{kd}}\big)=v_{\mathfrak p}(a_d(p^k))\ge0$, and dividing by $p^{kd}>0$ gives us that $v_{\mathfrak p}(\kappa_d)\ge0$.\\

For part 2, fix a prime $\mathfrak p\mid p$ of $K$, let $f=f(\mathfrak p/p)$ be its inertia degree, so the residue field is $ k_\mathfrak{p}\cong\F_{p^f}$, of multiplicative order $p^f-1$, that is coprime to $p$. By part 1, $\kappa_d$ has a well-defined reduction $\bar\kappa_d\in k_\mathfrak{p}$, since reduction mod $\mathfrak p$ is a ring homomorphism on $\mathfrak p$-integral elements, we have that $\bar\kappa_d^{\,p^{kd}}\ =\ \overline{\kappa_d^{p^{kd}}}\ =\ \overline{a_d(p^k)}$. By Theorem~\ref{prop:Gauss} and Corollary~\ref{cor:p_mod_1}, $a_d(p^k)\equiv\kappa_d^{(p)}\equiv1\pmod p$, as $p\in\mathfrak p$, this congruence mod $p$ (as rational integers) implies $\overline{a_d(p^k)}=1$ in $ k_\mathfrak{p}$ as well. So $\bar\kappa_d^{\,p^{kd}}=1$. In particular $\bar\kappa_d\ne0$ (else the left side would be $0$), so $\bar\kappa_d\in k_{\mathfrak p}^\times$, a cyclic group of order $p^f-1$. Since $\gcd(p^{kd},p^f-1)=1$ (as $p^f-1\equiv-1\pmod p$), the map $x\mapsto x^{p^{kd}}$ is a group automorphism of $k_{\mathfrak p}^\times$, the unique solution of $x^{p^{kd}}=1$ in this group is therefore $x=1$. Hence $\bar\kappa_d=1$, i.e.\ $\kappa_d\equiv1\pmod{\mathfrak p}$.
\end{proof}

\begin{remark}
    We can provide an alternative proof of part 2 of Theorem~\ref{thm:localrigidity} based upon the results from Section~\ref{sec:padic}. Fixing some $\mathfrak p\mid p$, since $a_d(p^k)=\kappa_d^{p^{kd}}$ is a rational integer, this equality holds in $K$, hence under the embedding $K\hookrightarrow K_\mathfrak p$. By Proposition~\ref{prop:kappa_p_exists}, $\kappa_d^{(p)}\equiv a_d(p^k)\pmod{p^k}$ in $\Zp$, i.e.\ $v_p\big(\kappa_d^{(p)}-\kappa_d^{p^{kd}}\big)\ge k$ in $\Zp$, extending scalars to $K_\mathfrak p$ (where $v_\mathfrak p|_{\Qp}=e_\mathfrak p\cdot v_p$), we get that  $v_\mathfrak p\big(\kappa_d^{(p)}-\kappa_d^{p^{kd}}\big)\ge e_\mathfrak p k\ge1>0$. Suppose toward a contradiction that $v_\mathfrak p(\kappa_d-1)\le0$. By cases 2 and 3 in Lemma~\ref{lem:dichotomy}, we have that $v_\mathfrak p\big(\kappa_d^{p^{kd}}-1\big)\le0$ for every $k\ge0$. By Proposition~\ref{prop:ANppinned}, if $p\in\Tame_d$ then $v_\mathfrak p(\kappa_d^{(p)}-1)=e_\mathfrak p r_p>0$ is fixed, more simply, directly from Proposition~\ref{prop:kappa_p_exists}, $\kappa_d^{(p)}\in1+p\Zp$ so $v_p(\kappa_d^{(p)}-1)\ge1$, i.e.\ $v_\mathfrak p(\kappa_d^{(p)}-1)\ge e_\mathfrak p\ge1$, in particular $v_\mathfrak p(\kappa_d^{(p)}-1)\ge1>0\ge v_\mathfrak p(\kappa_d^{p^{kd}}-1)$, so by the ultrametric inequality $v_\mathfrak p\big(\kappa_d^{(p)}-\kappa_d^{p^{kd}}\big)=v_\mathfrak p\big(\kappa_d^{p^{kd}}-1\big)\le0$ for every $k$. Taking $k$ as in the hypothesis, this contradicts $v_\mathfrak p(\kappa_d^{(p)}-\kappa_d^{p^{kd}})\ge e_\mathfrak p k\ge1>0$.
\end{remark}

\begin{corollary}\label{cor:normbound}
    If $p^k \in E_d$ then $p\le|N_{K/\Q}(\kappa_d-1)|$.
\end{corollary}

\begin{proof}
    By Theorem~\ref{thm:localrigidity} we have that $\kappa_d\equiv1\pmod{\mathfrak p}$ for every prime $\mathfrak p\mid p$ of $K=\Q(\kappa_d)$, and in particular $\mathfrak p\mid(\kappa_d-1)$ for at least one such $\mathfrak p$. Thus, $p=N(\mathfrak p)^{1/f}\mid N_{K/\Q}(\kappa_d-1)$ divides the norm (as the ideal $(\kappa_d-1)$ is divisible by $\mathfrak p$, and so, its norm, which equals to $|N_{K/\Q}(\kappa_d-1)|$, is divisible by $N(\mathfrak p)=p^f$). Since $\kappa_d\ne1$ (as $\kappa_d\in[\sqrt2,\sqrt3]$ from Lemma~\ref{lem:naive_bound}) we have that $\kappa_d-1$ is a nonzero algebraic integer, and so $N_{K/\Q}(\kappa_d-1)$ is a nonzero rational integer divisible by $p$, giving us that $p\le|N_{K/\Q}(\kappa_d-1)|$ as desired.
\end{proof}

\begin{corollary}\label{thm:normbound}
    Assume that $p ^k \in E_d$ and assume that $|\alpha|<R$ for every conjugate $\alpha \neq \kappa_d$ of $\kappa_d$. Then $$p \leq \big(\big(2^{2d+1}-1\big)^{1/4d}-1\big)(R+1)^{D-1}.$$ 
\end{corollary}

\begin{proof}
    $N_{K/\Q}(\kappa_d-1)=\prod_\sigma(\sigma(\kappa_d)-1)$, where the product is over all $D$ embeddings $\sigma:K\to\C$. From Theorem~\ref{thm:decreasing} we have that $|\kappa_d-1|\le \big(2^{2d+1}-1\big)^{1/4d}-1$. Every other embedding sends $\kappa_d$ to a conjugate of modulus $\le R$, and so contributes $|\sigma(\kappa_d)-1|\le|\sigma(\kappa_d)|+1\le R+1$. Hence $|N_{K/\Q}(\kappa_d-1)|\le(\sqrt3-1)\cdot (R+1)^{D-1}$, and the result follows from Corollary~\ref{cor:normbound}.
\end{proof}

\begin{corollary}\label{cor:modprigid}
Suppose $\kappa_d$ is algebraic with minimal polynomial $f$. Suppose $p^k\in E_d$ for some $k\ge1$ and some prime $p$. Then:
\begin{enumerate}
    \item $f(x)\equiv(x-1)^D \pmod p$.
    \item $p$ divides the discriminant of the polynomial $f$,
    \item If, in addition,  $p\nmid[\mathcal O_K:\Z[\kappa_d]]$ then $p$ is totally ramified in $K$, i.e. there exists a unique prime $\mathfrak{p} \mid p$. 
\end{enumerate}
\end{corollary}
\begin{proof}
From Theorem~\ref{thm:localrigidity} we have that  $\kappa_d\equiv1\pmod{\mathfrak p}$ for every prime $\mathfrak p$ of $K=\Q(\kappa_d)$ above $p$.
Let $A=\Z[\kappa_d]\cong\sfrac{\Z[x]}{\langle f\rangle}$.  Every irreducible factor $g$ of $\overline f\in\F_p[x]$ determines a maximal ideal of $A$ above $p$. Since $\mathcal{O}_K$ is integral over $A$, the lying-over theorem supplies a prime $\mathfrak p$ of $\mathcal{O}_K$ above that maximal ideal. The reduction of $\kappa_d$ at $\mathfrak p$ is $1$. Therefore the corresponding irreducible factor is $x-1$. Thus, $x-1$ is the only irreducible factor of $\overline f$. Since (up to a constant) $f$ is monic of degree $D$, we can conclude that $\overline f=(x-1)^D$. Since $p^k \in E_d$ then we have that $\kappa_d$ is an algebraic integer, but as   $\kappa_d \in [\sqrt{2}, \sqrt{3}]$ by Lemma~\ref{lem:naive_bound}), then we must have that $D>1$. Therefore $\overline f$ has a repeated root, and the discriminant of $f$ vanishes modulo $p$. Finally, for item 3, from the Dedekind-Kummer theorem (which is valid since $p\nmid[\mathcal O_K:\Z[\kappa_d]]$) we have that, $f(x)\bmod p$ factors as $\prod_ig_i(x)^{e_i}$ where $p\mathcal O_K=\prod_i\mathfrak p_i^{e_i}$ and $g_i$ is the minimal polynomial over $\F_p$ of the residue $\bar\kappa_d\in\mathcal O_K/\mathfrak p_i$. But since $f(x)\equiv (x-1)^D\pmod p$, we can conclude that  $\sum_ie_if_i=D$ and $f_i=1$ for every $i$, and the result follows.
\end{proof}

\begin{remark}
We can use Theorem~\ref{thm:localrigidity} to reprove Proposition~\ref{thm:finiteexp} (in the case where $p$ is odd). Suppose $p^k\in E_d$ for some $k$ in an infinite set $S \subseteq \mathbb{N}$. Fix any $k_1\in S$, then Theorem~\ref{thm:localrigidity} gives us that $v_\mathfrak p(\kappa_d-1)>0$ for every $\mathfrak p\mid p$, i.e.\ we are in Case 1 of Lemma~\ref{lem:dichotomy} at every such $\mathfrak p$. Fixing some $\mathfrak p\mid p$, by case 1 of Theorem~\ref{thm:comparison}, there is $k_0$ such that $v_\mathfrak p\big(\kappa_d^{(p)}-\kappa_d^{p^{kd}}\big)=e_\mathfrak p r_p$ for every $k\ge k_0$, where $r_p=\nu_p(a_d(p)-1)$ is fixed (independent of $k$). But for $k\in S$, as in the proof of Theorem~\ref{thm:localrigidity}, $v_\mathfrak p\big(\kappa_d^{(p)}-\kappa_d^{p^{kd}}\big)\ge e_\mathfrak pk$. Combining, for $k\in S$ with $k\ge k_0$: $e_\mathfrak pk\le e_\mathfrak pr_p$, that is $k\le r_p$. Since $S$ is infinite, it contains some $k\ge\max(k_0,r_p+1)$, and so we have that$k\le r_p<k$, which is a contradiction.
\end{remark}

We end this section with a discussion on $E_d$ in the case where $\kappa_d$ belongs to one of three special families of algebraic numbers, the first being algebraic units:

\begin{proposition}\label{cor:quadunit}
Assume that $\kappa_d$ is an algebraic integer and a unit in its ring of integers $\mathcal O_{K}$. Then:
\begin{enumerate}
    \item $E_d=\emptyset$.
    \item If $D=2$ then $d=1$.
\end{enumerate}

\end{proposition}
\begin{proof}
For part 1, suppose $n_0\in E_d$, so $a_d(n_0)=\kappa_d^{n_0^d}$.Since $\kappa_d$ is a unit, then by taking norms over $\Q$ gives us that $N_{\Q(\kappa_d)/\Q}(\kappa_d)=\pm1$, and so $N(\kappa_d^{n_0^d})=N(\kappa_d)^{n_0^d}=\pm1$. On the other hand,  $\kappa_d^{n_0^d}=a_d(n_0)$ is already a rational integer, and so fixed by every embedding of $\Q(\kappa_d)$ into $\C$. This gives us that $N(\kappa_d^{n_0^d})=a_d(n_0)^D$. Therefore we get that $a_d(n_0)^D=\pm1$, which tells us that $a_d(n_0)=1$ (as $a_d(n_0)>0$), which contradicts Lemma~\ref{lem:betaheight} that tells us that $a_d(n_0)\ge2$.\\

For part 2, since $\kappa_d$ is a unit, its norm is $\pm1$, and so we can write its minimal polynomial as $x^2-bx+c\in\Z[x]$ where $c=\pm1$ (since $c$ is exactly the norm of $\kappa_d$ in this case). Since $\kappa_d$ is a quadratic unit then the other conjugate of $\kappa_d$ is $\frac{N_{K/\Q}(\kappa_d)}{\kappa_d} = \pm \frac{1}{\kappa_d}$, and by Vieta's formula we have that $b=\kappa_d \pm \frac{1}{\kappa_d}$. Since the functions $f_{\pm}(x)=x\pm\frac{1}{x}$ are increasing on the interval $[\sqrt{2}, \sqrt{3}]$, we have that $b \in [\sqrt{2}\pm\frac{1}{\sqrt{2}}, \sqrt{3}\pm\frac{1}{\sqrt{3}}]$. The plus case is impossible since $b$ is an integer but there are no integers in the interval $[\sqrt{2}+\frac{1}{\sqrt{2}}, \sqrt{3}+\frac{1}{\sqrt{3}}]$, so we can conclude that $c=-1$ and so $b=1$. Therefore the minimal polynomial of $\kappa_d$ is $x^2-x-1$, and so from Lemma~\ref{lem:naive_bound} and from Theorem~\ref{thm:decreasing} we must have that $d=1$ (noting that $\tfrac{1+\sqrt5}2$ is indeed a quadratic unit as a solution to the equation $x(x-1)=1$).
\end{proof}

\begin{remark}
For $d=1$, Proposition~\ref{cor:quadunit} correctly predicts $E_1=\emptyset$, as $a_1(n)=L_n\in\Z$ while $\kappa_1^n\notin\Z$ for $n\ge1$, so $a_1(n)\ne\kappa_1^n$ always. %More generally, Proposition~\ref{cor:quadunit} shows that if some $\kappa_d$ with $d\ge2$ were ever found to be an algebraic-integer unit - the natural higher-dimensional analogue of the $d=1$ picture - then (unlike at $d=1$, where the relevant "coincidence-freeness" is a triviality about integrality) $E_d$ would be provably empty, and none of Theorem~\ref{thm:coprimality} or~\ref{thm:finiteexp} would have any content (both being statements about nonempty $E_d$). This is consistent with, and a bad case of, the general rigidity picture: an algebraic-integer unit is exactly a $\kappa_d$ with "no room" for the norm computation in Theorem~\ref{thm:coprimality}'s proof to produce an integer of the required size, in the sharpest possible way.
\end{remark}

We now look at two other special families of algebraic numbers, namely the Pisot numbers and the Salem numbers. 

\begin{definition}
Let $\alpha>1$ be a real algebraic integer of degree $g$ with Galois conjugates $\alpha = \alpha_1, \dots, \alpha_g$.  
    \begin{itemize}
        \item We say that $\alpha$ is a \textbf{Pisot Number} if $|\alpha_i|<1$ for every $i>1$.
        \item We say that $\alpha$ is a \textbf{Salem Number} if $g \geq 4$ is even, $\alpha_i=\alpha^{-1}$ for some $i$, and $|\alpha_j|=1$ for every $j \notin\{1,i\}$.
    \end{itemize}
\end{definition}

\begin{proposition}
If $\kappa_d$ is a quadratic Pisot number, $E_d$ contains no prime power.
\end{proposition}
\begin{proof}
    The result follows from Corollary~\ref{thm:normbound}, noting that we can choose $R=1$. 
\end{proof}

%(A real quadratic Salem number is impossible: its single other conjugate would need modulus exactly $1$ while being a real algebraic number, forcing it to be $\pm1$, but then $\kappa_d$ would be a root of $X^2\mp X-c$ for the corresponding rational integer trace/norm relation making it rational - contradicting irrationality of $\kappa_d$ on $[\sqrt2,\sqrt3]$, or more simply: a real quadratic algebraic integer's two conjugates are real, and modulus exactly $1$ for a real number just means the conjugate is $\pm1\in\Z$, forcing $\kappa_d\in\Z$ too via the trace being an integer, again impossible.)

\begin{proposition}\label{prop:salem}
If $\kappa_d$ is a Salem number of degree $D$ and $p^k \in E_d$,
\[
p\ <\ \Bigg( \big(2^{2d+1}-1\big)^{1/4d} + \frac{1}{\big(2^{2d+1}-1\big)^{1/4d}} -2\Bigg)\cdot2^{\,D-2}.
\]
In particular, if $\kappa_d$ is a Salem number of degree $4$ then $p^k \notin E_d$ for every prime $p$ and for every $k$.
\end{proposition}

\begin{proof}
Since $\kappa_d$ is assumed to be a Salem number, then we can denote the conjugates of $\kappa_d$ by $\kappa_d, \frac{1}{\kappa_d}, e^{i\theta_1}, e^{-i\theta_1}, \dots, e^{i\theta_{\frac{D-2}{2}}}, e^{-i\theta_{\frac{D-2}{2}}}$. Therefore we have that $N_{K/\mathbb Q}(\kappa_d)=\kappa_d\cdot\tfrac1{\kappa_d}\cdot\prod_je^{i\theta_j}e^{-i\theta_j}=1$. For $\kappa_d-1$ we have that 
\[
N_{K/\mathbb Q}(\kappa_d-1)=(\kappa_d-1)\Big(\tfrac1{\kappa_d}-1\Big)\prod_{j=1}^{\frac{D-2}{2}}
(e^{i\theta_j}-1)(e^{-i\theta_j}-1) = -\frac{(\kappa_d-1)^2}{\kappa_d}
\prod_{j=1}^{\frac{D-2}{2}}(2-2\cos\theta_j),
\]
using $(e^{i\theta}-1)(e^{-i\theta}-1)=|e^{i\theta}-1|^2=2-2\cos\theta\ge0$. Therefore, since $\frac{(\kappa_d-1)^2}{\kappa_d}=\kappa_d-2+\tfrac1{\kappa_d}$, we have that 
\[
|N_{K/\mathbb Q}(\kappa_d-1)| = \left(\kappa_d-2+\tfrac1{\kappa_d}\right)\cdot \prod_{j=1}^{\frac{D-2}{2}}(2-2\cos\theta_j),
\]
Since the function $x \mapsto x-2+\frac{1}{x}$ is increasing on $[\sqrt2,\sqrt3]$, since for every $j$ we have that $2-2\cos\theta_j\le4$, and since from Theorem~\ref{thm:decreasing} we have that $|\kappa_d-1|\le \big(2^{2d+1}-1\big)^{1/4d}-1$, we can conclude that
\[
|N_{K/\mathbb Q}(\kappa_d-1)| \le \left(\kappa_d-2+\tfrac1{\kappa_d}\right)\cdot4^{\frac{D-2}{2}} \leq 
\Bigg( \big(2^{2d+1}-1\big)^{1/4d} + \frac{1}{\big(2^{2d+1}-1\big)^{1/4d}} -2\Bigg)\cdot2^{\,D-2},
\]
and the result follows from Corollary~\ref{cor:normbound}. Finally, if $D=4$ then $\Big(\tfrac{4\sqrt3}3-2\Big)\cdot2^{\,2} <2$, so $|N_{K/\mathbb Q}(\kappa_d-1)|=1$ which is divisible by no prime.
\end{proof}

\bibliographystyle{alpha}
\bibliography{bib}

@article{Arnold2006,
  author  = {Arnold, V. I.},
  title   = {On the matricial version of {F}ermat--{E}uler congruences},
  journal = {Japan. J. Math.},
  volume  = {1},
  number  = {1},
  pages   = {1--24},
  year    = {2006}
}

@incollection{deninger2010p,
  title={p-adic entropy and ap-adic Fuglede--Kadison determinant},
  author={Deninger, Christopher},
  booktitle={Algebra, Arithmetic, and Geometry: Volume I: In Honor of Yu. I. Manin},
  pages={423--442},
  year={2010},
  publisher={Springer}
}

@misc{PrimeGrid,
  title = {PrimeGrid, End of WW Project—30 day notice, official project announcement, 29–30},
  howpublished = {\url{https://www.primegrid.com/forum_thread.php?id=10037&nowrap=true}},
  note = {December 2022}
}

@article{baxter1999planar,
  title={Planar lattice gases with nearest-neighbor exclusion},
  author={Baxter, RJ},
  journal={Annals of Combinatorics},
  volume={3},
  number={2},
  pages={191--203},
  year={1999},
  publisher={Springer}
}

@article{liang2025independent,
  title={Independent Set Enumeration and Estimation of Related Constants of Grid Graphs and Their Variants},
  author={Liang, Kai},
  journal={arXiv preprint arXiv:2507.04007},
  year={2025}
}

@incollection{GeorgiiHaggstromMaes2001,
  author    = {Georgii, Hans-Otto and H{\"a}ggstr{\"o}m, Olle and Maes, Christian},
  title     = {The random geometry of equilibrium phases},
  booktitle = {Phase Transitions and Critical Phenomena},
  editor    = {Domb, C. and Lebowitz, J. L.},
  volume    = {18},
  pages     = {1--142},
  publisher = {Academic Press},
  year      = {2001},
  doi       = {10.1016/S1062-7901(01)80008-2}
}

@article{HuChen1991,
  author  = {Hu, Chin-Kun and Chen, Chi-Ning},
  title   = {Percolation renormalization-group approach to the hard-square model},
  journal = {Physical Review B},
  volume  = {43},
  number  = {7},
  pages   = {6184--6185},
  year    = {1991},
  doi     = {10.1103/PhysRevB.43.6184}
}

@article{LiuEvans2000,
  author  = {Liu, Da-Jiang and Evans, J. W.},
  title   = {Ordering and percolation transitions for hard squares:
             Equilibrium versus nonequilibrium models for adsorbed layers
             with superlattice ordering},
  journal = {Physical Review B},
  volume  = {62},
  number  = {3},
  pages   = {2134--2145},
  year    = {2000},
  doi     = {10.1103/PhysRevB.62.2134}
}

@article{vanDenBergSteif1994,
  author  = {van den Berg, J. and Steif, J. E.},
  title   = {Percolation and the hard-core lattice gas model},
  journal = {Stochastic Processes and their Applications},
  volume  = {49},
  number  = {2},
  pages   = {179--197},
  year    = {1994},
  doi     = {10.1016/0304-4149(94)90132-5}
}

@article{vanDenBergMaes1994,
  author  = {van den Berg, J. and Maes, C.},
  title   = {Disagreement percolation in the study of {Markov} fields},
  journal = {The Annals of Probability},
  volume  = {22},
  number  = {2},
  pages   = {749--763},
  year    = {1994},
  doi     = {10.1214/aop/1176988728}
}

@article{ban2021topological,
  title={On the topological entropy of subshifts of finite type on free semigroups},
  author={Ban, Jung-Chao and Chang, Chih-Hung},
  journal={Taiwanese Journal of Mathematics},
  volume={25},
  number={3},
  pages={579--598},
  year={2021},
  publisher={Mathematical Society of the Republic of China}
}

@article{baxter1980hard,
  title={Hard-square lattice gas},
  author={Baxter, Rodney J and Enting, IG and Tsang, SK},
  journal={Journal of Statistical Physics},
  volume={22},
  number={4},
  pages={465--489},
  year={1980},
  publisher={Springer}
}

@article{byszewski2021dold,
  title={Dold sequences, periodic points, and dynamics},
  author={Byszewski, Jakub and Graff, Grzegorz and Ward, Thomas},
  journal={Bulletin of the London Mathematical Society},
  volume={53},
  number={5},
  pages={1263--1298},
  year={2021},
  publisher={Wiley Online Library}
}

@article{Zarelua2006,
  author  = {Zarelua, A. V.},
  title   = {On matrix analogs of {F}ermat's little theorem},
  journal = {Mat. Zametki},
  volume  = {79},
  number  = {6},
  pages   = {840--855},
  year    = {2006},
  note    = {English translation: Math. Notes \textbf{79} (2006), no.\ 5--6, 783--796}
}

@article{Zarelua2008,
  author  = {Zarelua, A. V.},
  title   = {On congruences for the traces of powers of some matrices},
  journal = {Proc. Steklov Inst. Math.},
  volume  = {263},
  pages   = {78--98},
  year    = {2008}
}

@article{CalkinWilf1998,
  author  = {Calkin, N. and Wilf, H.},
  title   = {The number of independent sets in a grid graph},
  journal = {SIAM J. Discrete Math.},
  volume  = {11},
  number  = {1},
  pages   = {54--60},
  year    = {1998}
}

@article{Dwork1960,
  author  = {Dwork, B.},
  title   = {On the rationality of the zeta function of an algebraic variety},
  journal = {Amer. J. Math.},
  volume  = {82},
  number  = {3},
  pages   = {631--648},
  year    = {1960}
}

@article{Borel1894,
  author  = {Borel, {\'E}.},
  title   = {Sur une application d'un th{\'e}or{\`e}me de {M}.\ {H}adamard},
  journal = {Bull. Sci. Math. (2)},
  volume  = {18},
  pages   = {22--25},
  year    = {1894}
}

@article{HM2010,
  author  = {Hochman, M. and Meyerovitch, T.},
  title   = {A characterization of the entropies of multidimensional shifts of finite type},
  journal = {Ann. of Math.},
  volume  = {171},
  number  = {3},
  pages   = {2011--2038},
  year    = {2010}
}

@article{Lind1984,
  author  = {Lind, D.},
  title   = {The entropies of topological {M}arkov shifts and a related class of algebraic integers},
  journal = {Ergodic Theory Dynam. Systems},
  volume  = {4},
  number  = {2},
  pages   = {283--300},
  year    = {1984}
}

@article{McIntoshRoettger,
  author  = {McIntosh, R. J. and Roettger, E. L.},
  title   = {A search for {F}ibonacci-{W}ieferich and {W}olstenholme primes},
  journal = {Math. Comp.},
  volume  = {76},
  number  = {260},
  pages   = {2087--2094},
  year    = {2007}
}

@book{Neukirch,
  author    = {Neukirch, J.},
  title     = {Algebraic Number Theory},
  publisher = {Springer-Verlag},
  address   = {Berlin},
  year      = {1999}
}

@article{Northcott1949,
  author  = {Northcott, D. G.},
  title   = {An inequality in the theory of arithmetic on algebraic varieties},
  journal = {Proc. Cambridge Philos. Soc.},
  volume  = {45},
  pages   = {502--509, 510--518},
  year    = {1949},
  note    = {}
}

@article{Pavlov2012,
  author  = {Pavlov, R.},
  title   = {Approximating the hard square entropy constant with probabilistic methods},
  journal = {Ann. Probab.},
  volume  = {40},
  number  = {6},
  pages   = {2362--2399},
  year    = {2012}
}

@article{PavlovSchraudner2015,
  author  = {Pavlov, R. and Schraudner, M.},
  title   = {Entropies realizable by block gluing $\mathbb{Z}^d$ shifts of finite type},
  journal = {J. Anal. Math.},
  volume  = {126},
  pages   = {113--174},
  year    = {2015}
}

@book{Serre,
  author    = {Serre, J.-P.},
  title     = {Local Fields},
  series    = {Graduate Texts in Mathematics},
  volume    = {67},
  publisher = {Springer-Verlag},
  address   = {Berlin},
  year      = {1979}
}

@article{Boyd1981,
  author  = {Boyd, D.},
  title   = {Speculations concerning the range of {M}ahler's measure},
  journal = {Canad. Math. Bull.},
  volume  = {24},
  pages   = {453--469},
  year    = {1981}
}

@article{Kahn2001,
  author  = {Kahn, J.},
  title   = {An entropy approach to the hard-core model on bipartite graphs},
  journal = {Combin. Probab. Comput.},
  volume  = {10},
  number  = {3},
  pages   = {219--237},
  year    = {2001}
}

@article{Zhao2010,
  author  = {Zhao, Y.},
  title   = {The number of independent sets in a regular graph},
  journal = {Combin. Probab. Comput.},
  volume  = {19},
  number  = {2},
  pages   = {315--320},
  year    = {2010}
}

@article{Pauling1935,
  author  = {Pauling, L.},
  title   = {The structure and entropy of ice and of other crystals with some randomness of atomic arrangement},
  journal = {J. Amer. Chem. Soc.},
  volume  = {57},
  number  = {12},
  pages   = {2680--2684},
  year    = {1935}
}

@article{Baxter1980,
  author  = {Baxter, R. J.},
  title   = {Hard hexagons: exact solution},
  journal = {J. Phys. A},
  volume  = {13},
  number  = {3},
  pages   = {L61--L70},
  year    = {1980}
}

@book{BaxterBook1982,
  author    = {Baxter, R. J.},
  title     = {Exactly Solved Models in Statistical Mechanics},
  publisher = {Academic Press},
  address   = {London},
  year      = {1982}
}

@inproceedings{Weitz2006,
  author    = {Weitz, D.},
  title     = {Counting independent sets up to the tree threshold},
  booktitle = {Proceedings of the Thirty-Eighth Annual ACM Symposium on Theory of Computing (STOC '06)},
  pages     = {140--149},
  year      = {2006}
}

@inproceedings{Sly2010,
  author    = {Sly, A.},
  title     = {Computational transition at the uniqueness threshold},
  booktitle = {Proceedings of the 2010 IEEE 51st Annual Symposium on Foundations of Computer Science (FOCS)},
  pages     = {287--296},
  year      = {2010}
}

@inproceedings{BrightwellWinkler2002,
  author    = {Brightwell, G. R. and Winkler, P.},
  title     = {Hard constraints and the {B}ethe lattice: adventures at the interface of combinatorics and statistical physics},
  booktitle = {Proceedings of the International Congress of Mathematicians, Vol.\ III},
  pages     = {605--624},
  year      = {2002}
}

@book{HornJohnson,
  author    = {Horn, R. A. and Johnson, C. R.},
  title     = {Matrix Analysis},
  edition   = {2nd},
  publisher = {Cambridge University Press},
  year      = {2013}
}

@book{Gouvea,
  author    = {Gouv{\^e}a, F. Q.},
  title     = {$p$-adic Numbers: An Introduction},
  edition   = {2nd},
  publisher = {Springer-Verlag},
  address   = {Berlin},
  year      = {1997}
}

@book{Marcus,
  author    = {Marcus, D. A.},
  title     = {Number Fields},
  publisher = {Springer-Verlag},
  address   = {New York},
  year      = {1977}
}

@article{galvin2011threshold,
  title={A threshold phenomenon for random independent sets in the discrete hypercube},
  author={Galvin, David},
  journal={Combinatorics, probability and computing},
  volume={20},
  number={1},
  pages={27--51},
  year={2011},
  publisher={Cambridge University Press}
}

@article{ordentlich2004independent,
  title={Independent sets in regular hypergraphs and multidimensional runlength-limited constraints},
  author={Ordentlich, Erik and Roth, Ron M},
  journal={SIAM Journal on Discrete Mathematics},
  volume={17},
  number={4},
  pages={615--623},
  year={2004},
  publisher={SIAM}
}

@article{louidor2013independence,
  title={Independence entropy of $\mathbb{Z}^d$-shift spaces},
  author={Louidor, Erez and Marcus, Brian and Pavlov, Ronnie},
  journal={Acta Applicandae Mathematicae},
  volume={126},
  number={1},
  pages={297--317},
  year={2013},
  publisher={Springer}
}

@article{galvin2019independent,
  title={Independent sets in the discrete hypercube},
  author={Galvin, David},
  journal={arXiv preprint arXiv:1901.01991},
  year={2019}
}

@article{sapozhenko1987number,
  title={On the number of connected subsets with given cardinality of the boundary in bipartite graphs},
  author={Sapozhenko, Aleksandr Antonovich},
  journal={Metody Diskret. Analiz},
  volume={45},
  number={45},
  pages={42--70},
  year={1987}
}

@article{korshunov1983number,
  title={The number of binary codes with distance 2},
  author={Korshunov, Aleksej D and Sapozhenko, Alexander A},
  journal={Problemy Kibernet},
  volume={40},
  number={111-130},
  pages={4},
  year={1983}
}

@article{JenssenPerkins2020,
  author  = {Jenssen, M. and Perkins, W.},
  title   = {Independent sets in the hypercube revisited},
  journal = {J. London Math. Soc.},
  volume  = {102},
  number  = {2},
  pages   = {645--669},
  year    = {2020}
}

@book{BombieriGubler,
  author    = {Bombieri, Enrico and Gubler, Walter},
  title     = {Heights in {D}iophantine Geometry},
  series    = {New Mathematical Monographs},
  volume    = {4},
  publisher = {Cambridge University Press},
  address   = {Cambridge},
  year      = {2006}
}

@book{WaldschmidtDA,
  author    = {Waldschmidt, Michel},
  title     = {Diophantine Approximation on Linear Algebraic Groups: Transcendence Properties of the Exponential Function in Several Variables},
  series    = {Grundlehren der mathematischen Wissenschaften},
  volume    = {326},
  publisher = {Springer},
  address   = {Berlin},
  year      = {2000}
}

\end{document}